\documentclass[a4paper,UKenglish,cleveref, autoref, thm-restate, nolineno]{socg-lipics-v2021}

\title{A Generalization of the Persistent Laplacian to Simplicial Maps} 

\author{Aziz Burak G\"ulen}{Department of Mathematics, The Ohio State University, Columbus, OH, USA \and \url{https://sites.google.com/view/azizburakgulen/} }{guelen.1@osu.edu}{}{}

\author{Facundo M\'emoli}{Department of Mathematics and Department of Computer Science and Engineering, The Ohio State University, Columbus, OH, USA \and \url{https://facundo-memoli.org/} } {memoli@math.osu.edu}{}{FM is partially supported by BSF 2020124, NSF CCF-1740761, NSF CCF-1839358, and NSF IIS-1901360.}

\author{Zhengchao Wan}{Halıcıoğlu Data Science Institute, University of California San Diego, San Diego, CA, USA \and \url{https://zhengchaow.github.io/} } {zcwan@ucsd,edu}{}{ZW is partially supported by NSF CCF-2112665, and NSF CCF-2217033.}

\author{Yusu Wang}{Halıcıoğlu Data Science Institute, University of California San Diego, San Diego, CA, USA \and \url{http://yusu.belkin-wang.org/} } {yusuwang@ucsd,edu}{}{YW is partially supported by NSF CCF-2112665, and NSF CCF-2217033.}

\authorrunning{A.\,B. G\"ulen, F. M\'emoli, Z. Wan and Y. Wang} 

\Copyright{Aziz Burak G\"ulen, Facundo M\'emoli, Zhengchao Wan, and Yusu Wang} 



\DeclareMathOperator{\Ima}{Im}
\DeclareMathOperator{\proj}{proj}
\newcommand{\Sch}{\mathbf{Sch}}
\newcommand{\sgn}{\mathrm{sgn}}
\newcommand{\ChUp}[2]{\mathfrak{C}_{q+1}^{#1\leftarrow #2}}
\newcommand{\ChDown}[2]{\mathfrak{C}_{q-1}^{#1\rightarrow #2}}

\usepackage{tikz-cd}

\newcommand{\R}        	{{\mathbb R}}
\newcommand{\N}        	{{\mathbb N}}

\usepackage{algorithm}
\usepackage{algorithmic}

\makeatletter
\newsavebox{\@brx}
\newcommand{\llangle}[1][]{\savebox{\@brx}{\(\m@th{#1\langle}\)}%
  \mathopen{\copy\@brx\mkern2mu\kern-0.9\wd\@brx\usebox{\@brx}}}
\newcommand{\rrangle}[1][]{\savebox{\@brx}{\(\m@th{#1\rangle}\)}%
  \mathclose{\copy\@brx\mkern2mu\kern-0.9\wd\@brx\usebox{\@brx}}}
\makeatother

\EventEditors{Erin W. Chambers and Joachim Gudmundsson}
\EventNoEds{2}
\EventLongTitle{39th International Symposium on Computational Geometry
(SoCG 2023)}
\EventShortTitle{SoCG 2023}
\EventAcronym{SoCG}
\EventYear{2023}
\EventDate{June 12--15, 2023}
\EventLocation{Dallas, Texas, USA}
\EventLogo{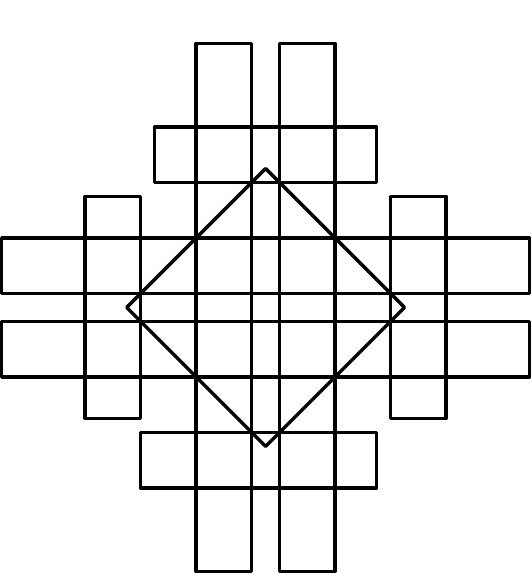}
\SeriesVolume{258}
\ArticleNo{39}  

\begin{document}

\maketitle

\begin{abstract}
    The (combinatorial) graph Laplacian is a fundamental object in the analysis of, and optimization on, graphs. Via a topological view, this operator can be extended to a simplicial complex $K$ and therefore offers a way to perform ``signal processing" on $p$-(co)chains of $K$.  Recently, the concept of \emph{persistent Laplacian} was proposed and studied for a pair of  simplicial complexes $K\hookrightarrow L$ connected by an inclusion relation, further broadening the use of Laplace-based operators. 
    
    In this paper, we significantly expand the scope of the persistent Laplacian by generalizing it to a pair of weighted simplicial complexes connected by a weight preserving simplicial map $f: K \to L$. Such a simplicial map setting arises frequently, e.g., when relating a coarsened simplicial representation with an original representation, or the case when the two simplicial complexes are spanned by different point sets, i.e. cases in which it does not hold that $K\subset L$.  However, the simplicial map setting is much more challenging than the inclusion setting since the underlying algebraic structure is much more complicated.

We present a natural generalization of the persistent Laplacian to the simplicial setting. To shed insight on the structure behind it, as well as to develop an algorithm to compute it, 
we exploit the relationship between the persistent Laplacian and the Schur complement of a matrix. A critical step is to view the Schur complement as a functorial way of restricting a self-adjoint positive semi-definite operator to a given subspace. As a consequence of this relation, we prove that the $q$th persistent Betti number of the simplicial map $f: K\to L$ equals  the nullity of the $q$th persistent Laplacian $\Delta_q^{K,L}$. We then propose an algorithm for finding the matrix representation of $\Delta_q^{K,L}$ which in turn yields  a fundamentally different algorithm for computing the $q$th persistent Betti number of a simplicial map. Finally, we study the persistent Laplacian on simplicial towers under weight-preserving simplicial maps and establish monotonicity results for their eigenvalues. 

\end{abstract}

\section{Introduction}

The graph Laplacian is an operator on the space of functions defined on the vertex set of a graph. It is one of the main tools in the analysis of and optimization on graphs. For example, the spectral properties of the graph Laplacian are extensively used in spectral clustering and other applications~\cite{specGraphTh, spielman2019spectral,LeeGT12,ng2002sca,uvon} and  for  efficiently solving  systems of equations~ \cite{KMP12,livne2012lean,spielman2004nearly,Vishnoi13}.

As opposed to the traditional way of defining the graph Laplacian as the difference of the degree matrix and the adjacency matrix, it can also be defined from an algebraic topology perspective by considering the boundary operators and specific inner products defined on simplicial chain groups~\cite{specGraphTh}. This point of view permits extending the graph Laplacian to operators on higher dimensional chain groups. Namely, this leads to the $q$th combinatorial Laplacian $\Delta_q^K$ on the $q$th chain group of a given simplicial complex $K$, in which the case $q=0$ corresponds to the standard graph Laplacian \cite{eckmann1944harmonische,duval2002shifted,goldberg2002combinatorial,horak2013spectra}.  One  fundamental property of the $q$th combinatorial Laplacian is that the $q$th Betti number of $K$ equals the nullity of $\Delta_q^K$.

By adopting the algebraic topology view, the $q$th persistent Laplacian $\Delta_q^{K,L}$ was independently introduced in~\cite{pdf, wang2020persistent} for a pair of simplicial complexes $K\hookrightarrow L$ connected by an inclusion. The theoretical properties of $\Delta_q^{K,L}$ and algorithms to compute it have been extensively studied in~\cite{memoli2022persistent}. One of these properties is that the nullity of $\Delta_q^{K,L}$ equals the \emph{persistent Betti number} of the inclusion $K\hookrightarrow L$, which is a generalization of the corresponding property of the combinatorial Laplacian mentioned above.

\begin{figure}[htbp]
    \centering
    \includegraphics[height=4cm]{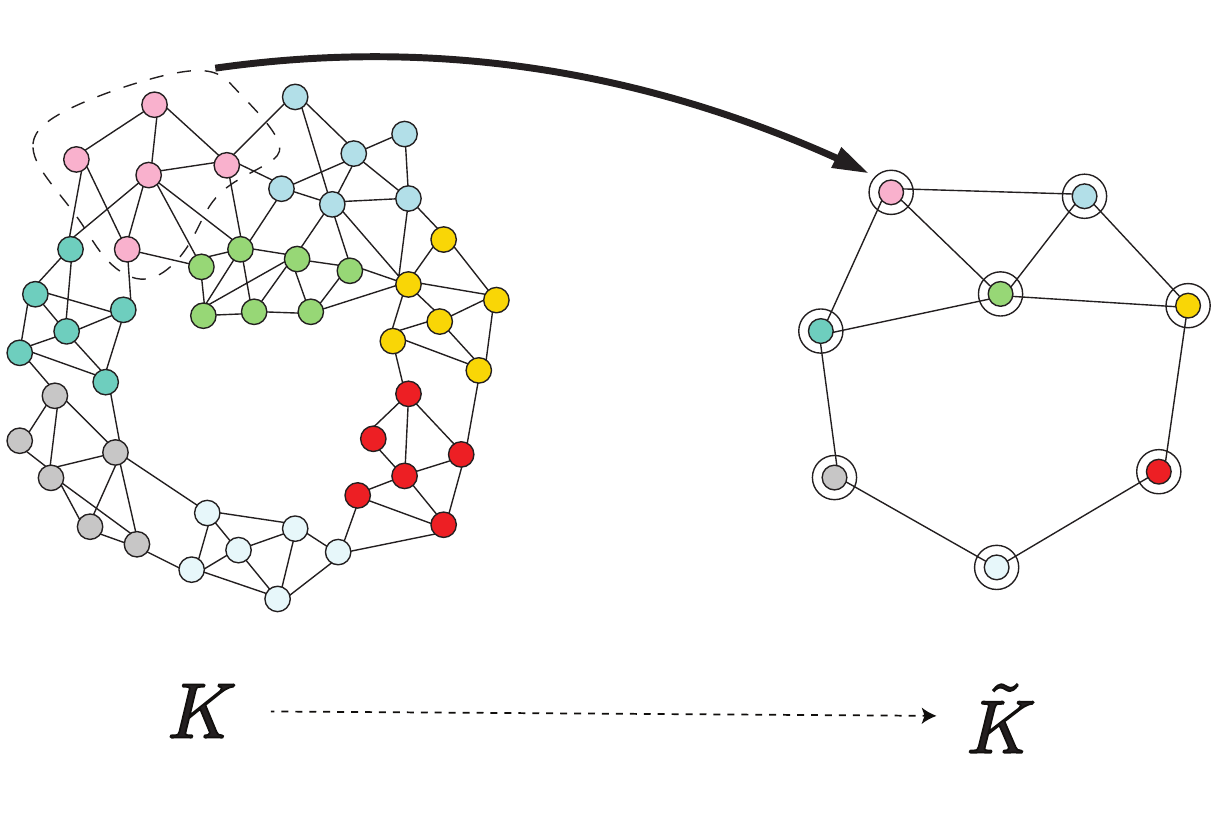}
    \caption{The 1-dimensional simplicial complex, i.e. graph, $K$ is coarsened to produce the one on the right $\tilde{K}$. Vertices of the same color are ``collapsed" to a ``supernode" in  $\tilde{K}$. This vertex map induces a simplicial map at the simplicial complex level.}
    \label{fig:coarsening}
\end{figure}
Although the persistent Laplacian for a pair $K\hookrightarrow L$ has been used in some applications~\cite{chen2022persistent, hoekzema2022multiscale, hozumi2022ccp}, the requirement that the complexes should be connected by an inclusion is restrictive and limits its applicability. Consider the scenario when we have two  simplicial complexes $K\stackrel{\iota}{\hookrightarrow} L$ related by an inclusion so that their sizes are prohibitively large. Instead of tackling the direct computation of the persistent Betti numbers induced by the simplicial inclusion $\iota$, practical needs may suggest that instead one \emph{sparsifies} the complexes $K$ and $L$ to obtain (smaller) complexes and in the process one obtains a \emph{simplicial map} connecting them (see Figure \ref{fig:coarsening} for an illustration of the coarsening procedure in the case of graphs). This is the scenario described for example in \cite{dey2013graph,dey2012computing} and can be expressed through the following  diagram where vertical arrows indicate the sparsification process: 

\begin{center}
\begin{tikzcd}

K \arrow[rr, "\iota"] \arrow[d,dashed] && L\arrow[d,dashed]\\
\tilde{K} \arrow[rr, "\varphi_\iota"] && \tilde{L}
  
\end{tikzcd}
\end{center}

This therefore motivates the study of persistent Laplacian for the setting where our input spaces (simplicial complexes) are connected by more general maps beyond inclusion, in particular, simplicial maps. This is the setting that we will study in this paper.

\paragraph*{Contributions}

We introduce a generalized version of the persistent Laplacian for weight preserving simplicial maps $f: K\to L$ between two weighted simplicial complexes $K$ and $L$. Our work utilizes ideas from several different disciplines, including operator theory, spectral graph theory, and persistent homology. In more detail:

\begin{itemize}
    \item In~\autoref{sec:basics-and-defns}, we provide two equivalent definitions of the (up and down) persistent Laplacian for a weight preserving simplicial map $f: K\to L$. While one definition is more useful when proving  some properties of the persistent Laplacian, the other definition provides a cleaner interpretation of the matrix representation of the persistent Laplacian. We also present one of the main properties of the persistent Laplacian,~\autoref{persBetti}, which establishes that the nullity of $\Delta_q^{f:K\to L}$ equals the persistent Betti number of the (arbitrary) simplicial map $f: K\to L$, analogous to the nonpersistent and the inclusion-based persistent cases.
    
    \item In~\autoref{sec:schur}, we show that the Schur complement of a principal submatrix in a matrix can be viewed as a (Schur) restriction of a self-adjoint positive semi-definite operator to a subspace. In order to accomplish this, we find it useful to utilize some concepts and language from category theory. Viewing the set of self-adjoint positive semi-definite operators as the poset category of the Loewner order\footnote{For two self-adjoint positive semi-definite operators $L_1$ and $L_2$, the Loewner order is given by: $L_1 \succeq L_2$ if and only if $L_1 - L_2$ is positive semi-definite.}, we prove that Schur restriction is a right adjoint to the functor that extends an operator on a subspace to the whole space by composing with projection onto that subspace. We present our core observation about the Schur restriction,~\autoref{thm:schur-persistent}, which states that up and down persistent Laplacians can be obtained via Schur restrictions of the combinatorial up and down Laplacians.
    
    \item In~\autoref{sec:matrix-and-algo}, we present an algorithm to find a matrix representation of the persistent Laplacian for simplicial maps by the relation between up/down persistent Laplacians and the Schur restriction. We also analyze its complexity. 
    
    \item In~\autoref{sec:mono}, we study the eigenvalues of up and down persistent Laplacians and prove monotonicity of these eigenvalues under the composition of simplicial maps.
\end{itemize}

\section{Persistent Laplacian for simplicial maps}
\label{sec:basics-and-defns}

\subsection{Basics}\label{sec:basics}
\subparagraph{Simplicial complexes and chain groups} An (abstract) simplicial complex $K$ over a finite ordered vertex set $V$ is a non-empty collection of non-empty subsets of $V$ with the property that for every $\sigma \in K$, if $\tau\subseteq \sigma$, then $\tau \in K$. An element $\sigma \in K$ is called a $q$-simplex if the cardinality of $\sigma$ is $q+1$. We denote the set of $q$-simplices by $S_q^K$.

An oriented simplex, denoted $[\sigma]$, is a simplex $\sigma \in K$ whose vertices are ordered. As we start with an ordered vertex set, we always assume that the orientation on the simplices are inherited from the order on the vertex set. Let $\mathcal{S}_q^K := \{ [\sigma]: \, \sigma \in K \}$.

The $q$th chain group $C_q^K := C_q(K, \R)$ of $K$ is the vector space over $\R$ with basis $\mathcal{S}_q^K$. Let $n_q^K := |\mathcal{S}_q^K| = \dim_\R (C_q^K)$.

The boundary operator $\partial_q^K : C_q^K \to C_{q-1}^K$ is defined by
\begin{equation}
    \partial_q^K ([v_0,...,v_q]) := \sum_{i=0}^{q}(-1)^i[v_0,...,\hat{v}_i,...,v_q]
\end{equation}
for every $q$-simplex $\sigma = [v_0,...,v_q]\in \mathcal{S}_q^K$, where $[v_0,...,\hat{v}_i,...,v_q]$ denotes the omission of the $i$th vertex, and extended linearly to $C_q^K$.

A weight function on a simplicial complex $K$ is any positive function $w^K : K\to (0, \infty)$. A simplicial complex is called weighted if it is endowed with a weight function. For every $q\in \N$, let $w_q^K := w^K|_{S_q^K}$, the restriction of $w^K$ onto $S_q^K$. We define an inner product $\langle \cdot, \cdot \rangle_{w_q^K}$ on $C_q^K$ as follows:
\begin{equation}
    \langle [\sigma], [\sigma'] \rangle_{w_q^K} := \delta_{\sigma \sigma'} \cdot (w_q^K(\sigma))^{-1} 
\end{equation}
for all $\sigma$, $\sigma' \in S_q^K$, where $\delta_{\sigma \sigma'}$ is the Kronecker delta.

\subparagraph{Cochain groups as dual of chain groups} For clarification of some of our results/notations later, we also introduce certain concepts related to cochain groups. The cochain group $C_K^q$ of $K$ is the linear space consisting of all linear maps defined on $C_q^K$, i.e., $C_K^q:=\mathrm{hom}(C_q^K,\R)$. The cochain group $C_K^q$ also possesses a natural basis $\mathcal{S}_K^q:=\{\chi_{[\sigma]}\,|\,[\sigma]\in \mathcal{S}_q^K\}$, where $\chi_{[\sigma]}$ is the linear map such that $\chi_{[\sigma]}([\tau])=\delta_{[\sigma],[\tau]}$ for any $[\tau]\in \mathcal{S}_q^K$. 
We define an inner product $\llangle\cdot,\cdot\rrangle_{w_q^K}$ on $C_K^q$ as follows: for any $\chi_{[\sigma]},\chi_{[\sigma']}\in \mathcal{S}_K^q$,
\begin{equation}
    \llangle \chi_{[\sigma]},\chi_{[\sigma']} \rrangle_{w_q^K} := \delta_{\sigma \sigma'} \cdot w_q^K(\sigma).
\end{equation}

Then, the map $j_q^K:C_q^K\rightarrow C_K^q$ sending a chain $c$ to the linear map $\langle c,\cdot\rangle_{w_q^K}$ is an isometry w.r.t. the inner products of the two spaces. Moreover, the following diagram commutes:
\begin{center}
\begin{tikzcd}

C_q^K \arrow[rr, "(\partial_{q+1}^K)^*"] \arrow[d,"j_q^K"'] && C_{q+1}^K\arrow[d,"j_{q+1}^K"]\\
C_K^q \arrow[rr, "\delta^q_K"]  && C_K^{q+1}
  
\end{tikzcd}
\end{center}
In this way, the adjoint $(\partial_{q+1}^K)^*$ of the boundary map $\partial_{q+1}^K$ can be identified with the coboundary map $\delta^q_K$. Similarly, $(\delta^{q}_K)^*$ can be identified with $\partial_{q+1}^K$. In the paper, we adopt the notation $L^*$ to denote the adjoint of a linear map $L$ between two inner product spaces.

\subparagraph{Combinatorial Laplacian for simplicial complexes} Given a weighted simplicial complex $K$, one defines the $q$th combinatorial Laplacian $\Delta_q^K$ as follows:
\[\Delta_q^K:={\partial_{q+1}^K\circ(\partial_{q+1}^K)^\ast}+{(\partial_{q}^K)^\ast\circ\partial_{q}^K} : C_q^K \to C_q^K,\]
where $\Delta_{q,\mathrm{up}}^K:=\partial_{q+1}^K\circ(\partial_{q+1}^K)^\ast$ is called the $q$th up Laplacian and $\Delta_{q,\mathrm{down}}^K:=(\partial_{q}^K)^\ast\circ\partial_{q}^K$ is called the $q$th down Laplacian.
Thanks to the renowned theorem by Eckmann \cite{eckmann1944harmonische}, the combinatorial Laplacian is able to capture topological information of underlying simplicial complexes: the nullity of $\Delta_q^K$ agrees with the $q$th Betti number of $K$.

\subparagraph{Simplicial maps}
A simplicial map from a simplicial complex $K$ into a simplicial complex $L$ is a function from the vertex set of $K$ to vertex set of $L$, $f: S_0^K \to S_0^L $, such that for every $\sigma \in K$, we have that $f(\sigma) \in L$. For every $q\in \N$, a simplicial map $f:K\to L$ induces a linear map $f_q : C_q^K \to C_q^L$ by the formula
\begin{equation}
    f_q([v_0,...,v_q]) =
    \begin{cases}
        [f(v_0),...,f(v_q)] & \text{if } f(v_0),...,f(v_q) \text{ are distinct} \\
        0 & \text{otherwise}
    \end{cases}
\end{equation}
for every oriented $q$-simplex $[v_0,...,v_q] \in \mathcal{S}_q^K$. 
The linear map $f_q$ does not have to preserve the orientation. That is, we could have that $f_q([\sigma]) = -[\tau]$ for some $[\sigma] \in \mathcal{S}_q^K$ and $[\tau] \in \mathcal{S}_q^L$. In this case, we write $\sgn_{f_q}(\sigma) = -1$. We write $\sgn_{f_q}(\sigma) = 1$ if $f_q([\sigma]) = [\tau]$. 

\begin{definition}
A simplicial map $f: K \to L$ between two weighted simplicial complexes is called \emph{weight preserving} if for every $[\tau] \in \Ima(f_q)$\, we have that 
\begin{equation}
    w_q^L(\tau) = \sum\limits_{\substack{\sigma \in S_q^K, \\ f_q([\sigma]) = \pm [\tau]}} w_q^K(\sigma). 
\end{equation}
\end{definition}

\subsection{The Persistent Laplacian for simplicial maps}\label{sec:perslap for simplicial}
The persistent Laplacian, whose definition we now recall,  was initially defined only for inclusion maps.
Given an inclusion map $\iota:K\hookrightarrow L$ between two simplicial complexes, we have the following commutative diagram 
\begin{center}
\begin{tikzcd}
                                         &                                                                                          & C_q^K \arrow[rr,red, "\partial_{q}^K"', bend right] \arrow[dd,gray, dashed, hook] \arrow[ld,blue, "{(\partial_{q+1}^{L,K})^*}", bend left] &  & C_{q-1}^K \arrow[ll,red, "(\partial_{q}^K)^*"', bend right] \\
                                         & {C_{q+1}^{L,K}} \arrow[ld,gray, dashed, hook] \arrow[ru, "{\partial_{q+1}^{L,K}}",blue, bend left] &                                                                                                                              &  &                                                         \\
C_{q+1}^L \arrow[rr, "\partial_{q+1}^L"] &                                                                                          & C_q^L                                                                                                                        &  &                                                        
\end{tikzcd}
\end{center}
Here, $C^{L,K}_{q+1}$ denotes the subspace $C^{L,K}_{q+1}:=\{c\in C_{q+1}^L\,|\,\partial_{q+1}^L(c)\in C_q^K\}$ of $C_{q+1}^L$, and $\partial_{q+1}^{L,K}$ denotes the restriction of $\partial_{q+1}^L$ to $C^{L,K}_{q+1}$, i.e., $\partial_{q+1}^{L,K}:=\partial_{q+1}^L|_{C^{L,K}_{q+1}}:C^{L,K}_{q+1}\rightarrow C_q^K$. Then, the $q$th up persistent Laplacian is defined as $\Delta_{q,\mathrm{up}}^{K,L}:=\partial_{q+1}^{L,K}\circ (\partial_{q+1}^{L,K})^*$, the $q$th down Laplacian is $\Delta_{q,\mathrm{down}}^K=(\partial_{q}^K)^\ast\circ\partial_{q}^K$, and the $q$th persistent Laplacian is defined as
\begin{equation}\label{eq:plap-inc}
\Delta_q^{K,L}:=\Delta_{q,\mathrm{up}}^{K,L}+\Delta_{q,\mathrm{down}}^{K} : C_q^K\to C_q^K.
\end{equation}
Similarly to the case of the combinatorial Laplacian, the nullity of $\Delta_q^{K,L}$ recovers the persistent Betti number of the inclusion map $\iota : K\hookrightarrow L$ (cf. \cite[Theorem 2.7]{memoli2022persistent}).

\subparagraph{Re-examination of the persistent Laplacian for inclusion maps}
Notice that (a) the definition of $C_{q+1}^{L,K}$ seems  to depend on the fact that the map $\iota$ is an inclusion and (b) the down Laplacian part $\Delta_{q,\mathrm{down}}^{K}$ does, a priori,  not exhibit any dependence on $L$. However, the apparent dependence/independence mentioned in (a) and (b), respectively, are illusory.
We now re-examine the definition above  in order to motivate our extension of the notion of  persistent Laplacian for  simplicial maps.

First of all, we note that the expression $\partial_{q+1}^L(c)\in C_q^K$ in the definition of $C_{q+1}^{L,K}$ above is somewhat misleading. In fact, we are implicitly identifying  $C_q^K$ with its image $ \iota_q(C_q^K)$ under the the inclusion map $\iota_q:C_q^K\rightarrow C_q^L$ induced by $\iota$. With this consideration, we rewrite $C^{L,K}_{q+1}$ in a more precise way:
\begin{equation}\label{eq:up inclusion}
    C^{L,K}_{q+1}=\left\{c\in C_{q+1}^L\,|\,\partial_{q+1}^L(c)\in \iota_q(C_q^K)\right\}.
\end{equation}

Expression (\ref{eq:up inclusion}) makes it clear that a certain set $\iota_q(C_q^K)$  is used in order to define the up Laplacian in the case of inclusions. This motivates us to consider the following dual construction which can be used to re-define the down Laplacian also in the case of inclusions
\begin{equation}\label{eq:down inclusion}
    \mathcal{C}^{K,L}_{q-1}:=\left\{c\in C_{q-1}^K\,|\,(\partial_{q}^K)^*(c)\in (\iota_q)^*(C_q^L)\right\}.
\end{equation}
As $\iota_q$ is injective, $(\iota_q)^*(C_q^L)=C_q^K$, and thus $\mathcal{C}^{K,L}_{q-1}=C_{q-1}^K$. 
In this way, we see that  using inclusion maps leads to concealing certain ``persistence-like'' structure inherent to the down part of the persistent Laplacian. An advantage of the formulation of the persistent Laplacian for general simplicial maps is that it will explicitly reveal this hidden structure.

Finally, we observe that for any $c\in C^{L,K}_{q+1}$, in fact, $\partial_{q+1}^L(c)\in \iota_q(\ker(\partial_q^K))\subseteq \iota_q(C_q^K)$. This is simply due to the fact that $\partial_q^K\circ \partial_{q+1}^L(c)=\partial_q^L\circ \partial_{q+1}^L(c)=0$. Here, we implicitly identify $\partial_{q+1}^L(c)$ with $\iota_q^{-1}(\partial_{q+1}^L(c))$ where $\iota^{-1}_q$ is the inverse of $\iota_q$ on its image. Hence, we have the following more refined expression for $C^{L,K}_{q+1}$:
\begin{equation}\label{eq:up inclusion alt}
    C^{L,K}_{q+1}=\left\{c\in C_{q+1}^L\,|\,\partial_{q+1}^L(c)\in \iota_q(\ker(\partial_q^K))\right\}.
\end{equation}

Integrating all these observations leads to our definition for the persistent Laplacian for general simplicial maps which we describe next.
\subparagraph{Persistent Laplacian for simplicial maps}
Suppose that we have a weight preserving simplicial map $f : K \to L$ and let $q\in \N$. Consider the subspaces
\[
\ChUp{L}{K} := \left\{ c\in C_{q+1}^L \mid \partial_{q+1}^L (c) \in f_q(\ker(\partial_q^K)) \right\},
\]
\[
\ChDown{K}{L} := \left\{ c\in C_{q-1}^K \mid (\partial_q^K)^*(c) \in \ker(f_q)^\perp \right\}.
\]
Note that $\ChUp{L}{K} \subseteq C_{q+1}^L$ and $\ChDown{K}{L}\subseteq C_{q-1}^K$. Moreover, these spaces are natural generalizations of $C^{L,K}_{q+1}$ and $\mathcal{C}^{K,L}_{q-1}$, respectively (cf. Equation \eqref{eq:up inclusion} and Equation \eqref{eq:down inclusion}), as $\ker(\iota_q)^\perp = C_q^K= (\iota_q)^*(C_q^L)$.

Let $\partial_{q+1}^{L,K}$ denote\footnote{The notation $\partial_{q+1}^{L,K}$ has been used before as the restriction of $\partial_{q+1}^L$ to $C_{q+1}^{L,K}$. As $\ChUp{L}{K}$ generalizes the space $C_{q+1}^{L,K}$, we stick to the same notation $\partial_{q+1}^{L,K}$ to denote the restriction of $\partial_{q+1}^L$ to $\ChUp{L}{K}$} the restriction of $\partial_{q+1}^L$ to $\ChUp{L}{K}$. 
Let $\delta_{q-1}^{K,L}$ denote\footnote{Recall that $(\partial_q^K)^*$ can be identified with the coboundary map $\delta^{q-1}_K$ in a sense specified in \autoref{sec:basics}, hence we use $\delta_{q-1}^{K,L}$ to denote this restriction.} the restriction of $(\partial_q^K)^*$ to $\ChDown{K}{L}$. Furthermore, we let $\hat{f}_q : \ker(f_q)^\perp \to \Ima(f_q)$ denote the restriction of $f_q$ onto $\ker(f_q)^\perp$. Before we proceed, we comment on some properties of $\hat{f}_q$ and $\ker(f_q)^\perp$.
We note that $\ker(f_q)^\perp$ possesses a canonical basis as follows. For every $[\tau] \in \Ima(f_q)$, we define
\[
    c_{q}^{\tau,f} := \sum\limits_{\substack{\sigma \in S_q^K, \\ f_q([\sigma]) = \pm [\tau]}} \, \sgn_{f_q}(\sigma)\,w_q^K(\sigma)\, [\sigma]\,\,\ \in C_q^K.
\]
When the map $f$ is clear from the content, we will simply write $c_q^{\tau}$. We let $ \mathcal{J} := \{ c_q^{\tau}\, |\, [\tau] \in \Ima(f_q)\}$.
\begin{lemma}\label{lem:isometry}
The set $\mathcal{J}$ is an orthogonal basis for $\ker(f_q)^\perp$. Moreover, the map $\hat{f}_q : \ker(f_q)^\perp \to \Ima(f_q)$ is an isometry between inner product spaces.
\end{lemma}

Now, we consider the following diagram which contains all the notations we defined above:
\begin{center}
\begin{tikzcd}
                                                                                &  & C_q^K                                                                                                                   &  & C_{q-1}^K                                                                        \\
                                                                                &  & \ker(f_q)^\perp \arrow[u, hook] \arrow[dd, "\hat{f}_q"', bend right] \arrow[rr, red, "{(\delta_{q-1}^{K,L})^*}", bend left]          &  & {\ChDown{K}{L}} \arrow[u, hook] \arrow[ll, red, "{\delta_{q-1}^{K,L}}"', bend left] \\
                                                                                &  &                                                                                                                         &  &                                                                                  \\
{\ChUp{L}{K}} \arrow[rr, blue, "{\partial_{q+1}^{L,K}}", bend left] \arrow[d, hook] &  & \Ima(f_q) \arrow[uu, "(\hat{f}_q)^{-1}"', bend right] \arrow[ll, blue, "{(\partial_{q+1}^{L,K})^*}", bend left] \arrow[d, hook] &  &                                                                                  \\
C_{q+1}^L                                                                       &  & C_q^L                                                                                                                   &  &                                                                                 
\end{tikzcd}\label{tik:laplacian}
\end{center}
We define up and down persistent Laplacian respectively as:
\begin{align}
    \Delta_{q,\mathrm{up}}^{K\stackrel{f}{\to} L} :=& \partial_{q+1}^{L,K}\circ (\partial_{q+1}^{L,K})^*:\Ima(f_q)\rightarrow \Ima(f_q), \label{defn:upLap} \\
    \Delta_{q,\mathrm{down}}^{K\stackrel{f}{\to} L} :=& \hat{f}_q \circ \delta_{q-1}^{K,L} \circ (\delta_{q-1}^{K,L})^* \circ \hat{f}_q^{-1}:\Ima(f_q)\rightarrow \Ima(f_q). \label{defn:downLap}
\end{align}

As $\hat{f}_q$ preserves inner product, we have that $\hat{f}_q^{-1} = \hat{f}_q^*$. Thus, both up and down persistent Laplacians are self-adjoint and non-negative operators on $\Ima(f_q)$. We then define the \emph{$q$-th persistent Laplacian} $\Delta_q^{K\stackrel{f}{\to} L} : \Ima(f_q) \to \Ima(f_q)$ by:
\begin{align}\label{eqn:perLdef}
\Delta_q^{K\stackrel{f}{\to} L} &:= \Delta_{q,\mathrm{down}}^{K\stackrel{f}{\to} L} + \Delta_{q,\mathrm{up}}^{K\stackrel{f}{\to} L}. 
\end{align}

When the map $f:K\to L$ is clear, we will write $\Delta_q^{K,L}$ for the persistent Laplacian.

\begin{remark}\label{rmk:image}
By slightly abuse of notation, we also let $f$ denote the simplicial map $f:K\rightarrow\Ima{(f)}$. Then, it follows from the definition of the down persistent Laplacian that $\Delta_{q,\mathrm{down}}^{K\stackrel{f}{\to} L} = \Delta_{q,\mathrm{down}}^{K\stackrel{f}{\to} \Ima{(f)}}$.
\end{remark}

\begin{remark}
When considering an inclusion $\iota :K\to L$, one can see that $\ChDown{K}{L} = \mathcal{C}_{q-1}^{K,L} = C_q^K$, $\ChUp{L}{K} = C_{q+1}^{L,K}$ and $\iota_q : C_q^K \hookrightarrow C_q^L$ is an isometric embedding. Thus, our definition of persistent Laplacian generalizes the inclusion-based persistent Laplacian
\end{remark}

\begin{remark}[An alternative definition of the persistent Laplacian]\label{rem:alternative-defn}
The weight preserving property of the simplicial map guarantees that $\ker(f_q)^\perp$ and $\Ima(f_q)$ are isometric, see~\autoref{lem:isometry}. Thus, we could have, equivalently, defined the (up and down) persistent Laplacian as an operator on $\ker(f_q)^\perp$ instead of $\Ima(f_q)$ as follows:
\begin{align*}
    \Delta_{q,\mathrm{up}}^{K\stackrel{f}{\to} L} :=& \hat{f}_q^{-1} \circ \partial_{q+1}^{L,K}\circ (\partial_{q+1}^{L,K})^* \circ \hat{f}_q:\ker(f_q)^\perp \rightarrow \ker(f_q)^\perp, \\
    \Delta_{q,\mathrm{down}}^{K\stackrel{f}{\to} L} :=& \delta_{q-1}^{K,L} \circ (\delta_{q-1}^{K,L})^* :\ker(f_q)^\perp \rightarrow \ker(f_q)^\perp.
\end{align*}
Note that when we have an inclusion $\iota :K \hookrightarrow L$, the (up/down) persistent Laplacian in~\cite{pdf,memoli2022persistent,wang2020persistent} is defined on $C_q^K$, which is the same as $\ker(\iota_q)^\perp$ and isometrically isomorphic to $\Ima(\iota_q)$.

The two different definitions have their own advantages. 
Seeing the persistent Laplacian as an operator on $\Ima(f_q)$ increases the interpretability of this operator as the matrix representation can be computed using the canonical basis of $\Ima(f_q)$. 
On the other hand, seeing the persistent Laplacian on $\ker(f_q)^\perp$ helps us understanding some of its properties more easily. 
For example, see proof of~\autoref{thm:down-mono}.
\end{remark}

\begin{remark}[Cochain formulation of the persistent Laplacian]\label{rmk:cochain laplacian}
Our generalization of the persistent Laplacian reveals a way to define a persistent Laplacian using the cochain spaces via dualization. 
If $f:K \to L$ is a simplicial map, then it induces a linear map in the cochain spaces $f^q : C^q_L \to C^q_K$, where $C_K^q = \mathrm{hom}(C_q^K, \R)$. Then, one can use the following subspaces in order to define a persistent Laplacian using cochains:
\[\mathfrak{C}_{L\leftarrow K}^{q+1} := \{ c\in C^{q+1}_L \mid (\delta_{q}^L)^*(c) \in (f^q)^*(\ker(\delta_{q-1}^K)^*)\},\]
\[\mathfrak{C}_{K\rightarrow L}^{q-1} := \{c\in C^{q-1}_K \mid \delta^K_{q-1}(c) \in \ker((f^q)^*)^\perp \}.\]
It turns out that the operator defined via these spaces are the same as the persistent Laplacian defined using chains; see~\autoref{sec:cochain laplacian} for more details.
\end{remark}

Let $\beta_q^{K\stackrel{f}{\to} L}$ denote the rank of the linear map $H_q(K)\to H_q(L)$ induced by $f$. $\beta_q^{K\stackrel{f}{\to} L}$ is called the persistent Betti number of the map $f:K\to L$. When the map $f:K\to L$ is clear from the content, we simply write $\beta_q^{K,L}$. With the machinery developed above together with several key observations that relates the (up and down) persistent Laplacians and Schur restriction of an operator, we have the following result.

\begin{theorem}[Persistent Laplacians recover persistent Betti numbers]\label{persBetti}
Let $f: K\to L$ be a simplicial map and $q \in \N$. Then, $\beta_q^{K,L} = \mathrm{nullity}(\Delta_q^{K,L})$.
\end{theorem}

\begin{remark}
As the persistent Betti number does not depend on the weights on the simplicial complexes, weights can be assigned to $K$ and $L$ such that the simplicial map $f:K \to L$ is weight preserving. Then, one can use the persistent Laplacian to compute the persistent Betti number of $f$.
\end{remark}

\section{Schur Restriction and the Persistent Laplacian}\label{sec:schur}
One of the main contributions in~\cite{memoli2022persistent} is a characterization of the up persistent Laplacian for inclusion maps via the so-called Schur complement. In this section, we establish that this characterization also holds in our setting of simplicial maps.

Let $M \in \R^{n \times n}$ be a block matrix $M = 
\begin{pmatrix}
A & B \\
C & D
\end{pmatrix}$
 where $A\in \R^{(n-d) \times (n-d)}$ and $D \in \R^{d \times d}$. The \emph{(generalized) Schur complement} of $D$ in $M$ is $M/D := A - B D^{\dagger} C$, where $D^{\dagger}$ is the Moore-Penrose generalized inverse of $D$.

A linear operator $L: V\to V$ on a finite dimensional real inner product space $V$ is called \emph{positive semi-definite} if $\langle L(v), v \rangle \geq 0$ for all $v\in V$, and it is called \emph{self-adjoint} if $L^* = L$. The Schur complement, more generally, can be seen as a way of restricting a self-adjoint positive semi-definite operator on a real inner product space onto a subspace as follows. 
Assume that $L : V\to V$ is a self-adjoint positive semi-definite opeator on $V$, where $V$ is a finite dimensional ($\dim_{\R} V = n$) real inner product space. 
Let $W \subseteq V$ be a $d$-dimensional subspace and let $W^{\perp}$ be its orthogonal complement. By choosing bases for $W$ and $W^{\perp}$, we can represent $L$ as a block matrix, say $[L] = \begin{pmatrix}
A & B \\
C & D
\end{pmatrix}$ where $A \in \R^{d \times d}$, $D \in \R^{(n-d) \times (n-d)}$. Then, $[L] / D = A-BD^\dagger C$ can be interpreted as the restriction of $L$ onto W, represented by the already chosen basis. We will see that the resulting operator represented by $[L]/D $ is independent of choice of basis (i.e. it is well-defined) and we call this operator the \emph{Schur restriction of $L$ onto $W$}, and denote it by $\Sch(L, W)$.

\begin{proposition}[The Schur restriction is well-defined]\label{prop:well-defn}
Let $L : V\to V$ be a self-adjoint positive semi-definite operator and let $W\subseteq V$ be a subspace. Then, $\Sch(L,W)$ is independent of choice of bases of $W$ and $W^\perp$. More explicitly, if $\mathcal{B}_1$ and $\mathcal{C}_1$ are ordered bases for $W$ and $\mathcal{B}_2$ and $\mathcal{C}_2$ are ordered bases for $W^\perp$, then the matrix representations of $\Sch(L,W)$ obtained from the ordered bases $\mathcal{B}_1 \cup \mathcal{B}_2$ and $\mathcal{C}_1 \cup \mathcal{C}_2$ are similar matrices via the change of basis matrix from $\mathcal{B}_1$ to $\mathcal{C}_1$. 
\end{proposition}

As~\autoref{prop:well-defn} guarantees that the Schur restriction of a self-adjoint positive semi-definite operator onto a subspace is well-defined, the next proposition reveals the recipe to acquire the Schur restriction and also justifies the name, ``\emph{Schur restriction}''.

\begin{proposition}\label{prop:schur-restriction}
Let $f: \hat{V} \to V$ be a linear map between two finite dimensional real inner product spaces and let $L = f\circ f^*:V\rightarrow V$. Let $W\subseteq V$ be a subspace. Let $f_W : f^{-1}(W) \to W$ be the restriction of $f$ on $f^{-1}(W)$ and the codomain is also restricted to $W$. Then, $\Sch(L, W) = f_W \circ f_W^*$.
\end{proposition}

The proof we present for~\autoref{prop:schur-restriction} in~\autoref{app:secSchur} heavily depends on the extremal characterization of Schur restrictions,~\autoref{exschur}, which essentially, in the language of category theory, states that Schur restriction, as a functor, is a right adjoint. See~\autoref{rem:adjoint} for details. One of the most significant applications of~\autoref{prop:schur-restriction} is the following theorem that establishes a relation between persistent Laplacians and the Schur restriction.

\begin{theorem}
[Up and down persistent Laplacians as Schur restrictions]\label{thm:schur-persistent}
For a weight-preserving simplicial map $f: K \to L$, we have that  
\begin{align*}
    &\Delta_{q,\mathrm{down}}^{K,L} = \hat{f}_q \circ \Sch(\Delta_{q,\mathrm{down}}^K, \ker(f_q)^\perp) \circ \hat{f}_q^{-1} \text{ and } \\
    &\Delta_{q,\mathrm{up}}^{K,L} = \iota_{\Ima(f_q)} \circ \Sch(\Delta_{q,\mathrm{up}}^L, f_q(\ker(\partial_{q}^K))) \circ \proj_{f_q(\ker(\partial_{q}^K))},
\end{align*}

where $\iota_{\Ima(f_q)} : f_q(\ker(\partial_q^K)) \hookrightarrow \Ima(f_q)$ is the inclusion map and $\proj_{f_q(\ker(\partial_{q}^K)} : \Ima(f_q) \to f_q(\ker(\partial_q^K))$ is the projection map.
\end{theorem}

\section{Matrix Representation of Persistent Laplacian and an Algorithm }\label{sec:matrix-and-algo}

Based on the Schur restriction characterization of persistent Laplacians, i.e.~\autoref{thm:schur-persistent}, in the previous section, we now derive an algorithm for computing the matrix representation of persistent Laplacians.

\subsection{Matrix Representation of Persistent Laplacian}

Let $f: K\to L$ be a weight preserving simplicial map. Recall that for every oriented $q$-simplex $[\tau] \in \Ima(f_q)$, we defined the $K$ $q$-chain
\[ 
c_{q}^{\tau} := \sum\limits_{\substack{\sigma \in S_q^K, \\ f([\sigma]) = \pm [\tau]}} \sgn_{f_q}(\sigma)w_q^K(\sigma) [\sigma]\,\,\ \in C_q^K.
\]

By~\autoref{lem:isometry}, the set $\mathcal{J} = \{ c_q^\tau \mid \tau\in \Ima(f_q) \}$
forms a orthogonal basis for $\ker(f_q)^\perp$.
Assume that $\{ [\tau_1] ,..., [\tau_n] \} \subseteq \Ima(f_q)$ is the set of all oriented $q$-simplices in $L$ that are hit by $f_q$.
Assume that for every $[\tau_i]$, $\{ [\sigma_1^i], ...., [\sigma_{d_i}^i] \} \subseteq \mathcal{S}_q^K$ is the set of all oriented $q$-simplices in $K$ that are mapped to $\pm [\tau_i]$. 
Define 
\[
\sigma^{i,k} := \sgn_{f_q}(\sigma_1^i)[\sigma_1^i] - \sgn_{f_q}(\sigma_k^i)[\sigma_k^i]
\]
for $i= 1, ..., n$ and $k = 2, ..., d_i$ for $d_i \geq 2$. 
Then, the set 
\[
\mathcal{B} = \{ \sigma^{i,k}  \mid 1\leq i \leq n \text{ , } 2\leq k \leq d_i\} \cup \{ [\sigma] \in \mathcal{S}_q^K \mid f_q([\sigma]) = 0 \}
\]
forms a basis for $\ker(f_q)$. Thus $\mathcal{J} \cup \mathcal{B}$ forms a basis for $C_q(K)$. Writing coordinates of basis elements of $\mathcal{J} \cup \mathcal{B}$ using the canonical basis $\mathcal{S}_q^K$ as column vectors, we obtain the change of basis matrix $M_{\mathcal{J}\cup \mathcal{B} \to \mathcal{S}_q^K}$.

\paragraph*{Matrix representation of down persistent Laplacian.} 
Let $[\Delta_{q,\mathrm{down}}^K]$ be the matrix representation of $\Delta_{q,\mathrm{down}}^K$ with respect to the canonical basis $\mathcal{S}_q^K$. Then, $N := (M_{\mathcal{J}\cup \mathcal{B} \to \mathcal{S}_q^K})^{-1} [\Delta_{q,\mathrm{down}}^K] M_{\mathcal{J}\cup \mathcal{B} \to \mathcal{S}_q^K}$ is the matrix representation of $\Delta_{q,\mathrm{down}}^K$ with respect to $\mathcal{J} \cup \mathcal{B}$. 
Given an integer $m$, let $[m]$ denote the set $[m] = \{1, 2, \ldots, m \}$. 
The matrix $N$ has dimension $n_q^K \times n_q^K$ where $n_q^K = |S_q^K|$. 
Let $n := |\mathcal{J}| = \dim(\Ima(f_q)) =\dim(\ker(f_q)^\perp)$ and let 
\begin{equation}\label{eq:XYZ}
    X = N([n],[n]),Y= N([n],[n_q^K]-[n]),Z = N([n_q^K]-[n], [n]), T = N([n_q^K]-[n],[n_q^K]-[n]).
\end{equation}
Then, we can write $N$ as a block matrix $N = \begin{pmatrix}
X & Y \\
Z & T
\end{pmatrix}$.
Let $W_{\Ima(f_q)}$ denote the diagonal matrix $W_{\Ima(f_q)} = \text{diag}(w(\tau_1),w(\tau_2),...,w(\tau_n))$. Then, we are now ready to write the matrix representation of $\Delta_{q,\mathrm{down}}^{K,L}$ with respect to the canonical basis $\{ [\tau_1],...,[\tau_n]\}$ of $\Ima(f_q)$.

\begin{proposition}\label{prop:downMatrix}
With the notations above, the matrix representation of $\Delta_{q,\mathrm{down}}^{K,L}$ with respect to the canonical basis $\{ [\tau_1],...,[\tau_n]\}$ of $\Ima(f_q)$ is given by 
\[
W_{\Ima(f_q)} (X - Y T^\dagger Z) W_{\Ima(f_q)}^{-1}.
\]
\end{proposition}

\begin{figure}[htb!]
    \centering
    \includegraphics{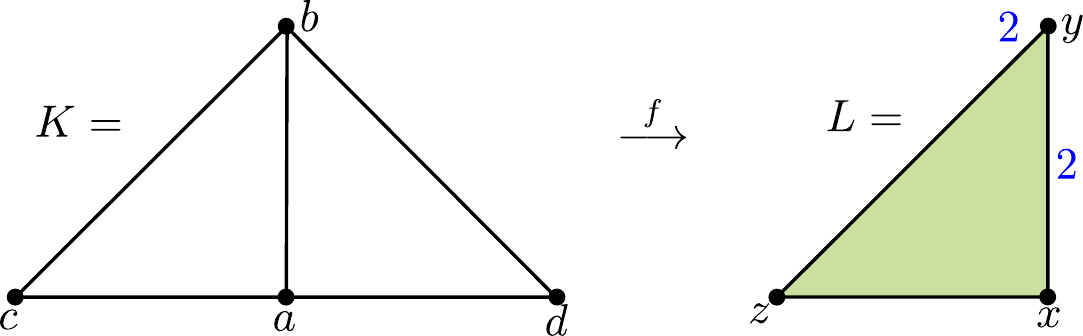}
    \caption{A weight preserving simplicial map $f : K\to L$ between two weighted simplicial complexes $K$ and $L$. $K$ has all the weights equal to 1. In $L$, the edge $xy$ and the vertex $y$ has weights $2$ and the rest of the simplicies have weight $1$. The map $f$ is given by $a\mapsto x$, $b\mapsto y$, $c\mapsto z$, $d\mapsto b$. And, ordering on the vertices are given by $a<b<c<d$ and $x<y<z$} 
    \label{fig:example}
\end{figure}
\begin{example}\label{ex:matrix rep down}
We will compute the matrix representation of the $1$st down persistent Laplacian of the weight preserving simplicial map depicted in~\autoref{fig:example}. The $1$st combinatorial down Laplacian of $K$ is given by
\[ [\Delta_{1,\mathrm{down}}^K] =
\begin{pmatrix}
2 & -1 & 1 & 1 & -1 \\
-1 & 2 & 1 & 0 & 1 \\
1 & 1 & 2 & 1 & 0 \\
1 & 0 & 1 & 2 & 1 \\
1 & 1 & 0 & 1 & 2
\end{pmatrix}.
\]
with respect to the canonical (ordered) basis $\mathcal{S}_1^K = \{[ab], [bc], [ac], [ad], [bd] \}$. Following the notation described above, we have that $\mathcal{J} = \{[ab]+[ad], [bc], [ac] \}$ and $\mathcal{B} = \{[ab]-[ad], [bd] \}$. Thus, we have the change of basis matrix as
\[
M_{\mathcal{J}\cup \mathcal{B} \to \mathcal{S}_1^K} =
\begin{pmatrix}
1 & 0 & 0 & 1 & 0 \\
0 & 1 & 0 & 0 & 0 \\
0 & 0 & 1 & 0 & 0 \\
1 & 0 & 0 & -1 & 0 \\
0 & 0 & 0 & 0 & 1
\end{pmatrix}.
\]
Then, we compute
\[
N = (M_{\mathcal{J}\cup \mathcal{B} \to \mathcal{S}_1^K})^{-1} [\Delta_{1,\mathrm{down}}^K] M_{\mathcal{J}\cup \mathcal{B} \to \mathcal{S}_1^K} =
\begin{pmatrix}
3 & -\frac{1}{2} & 1 & 0 & 0 \\
-1 & 2 & 1 & -1 & 1 \\
2 & 1 & 2 & 0 & 0 \\
0 & -\frac{1}{2} & 0 & 1 & -1 \\
0 & 1 & 0 & -2 & 2 
\end{pmatrix}.
\]
Now, by extracting $X, Y, Z,$ and $T$ as described above in Equation \eqref{eq:XYZ}, and realizing that $W_{\Ima(f_1)} =\mathrm{diag}(2,1,1)$, we write the matrix representation of the $1$st down persistent Laplacian $\Delta_{1,\mathrm{down}}^{K,L}$ with respect to the basis $\{[xy], [yz], [xz]\}$ as follows
\[
[\Delta_{1,\mathrm{down}}^{K,L}] = W_{\Ima(f_q)} (X - Y T^\dagger Z) W_{\Ima(f_q)}^{-1} =
\begin{pmatrix}
3 & -1 & 2\\
-\frac{1}{2} & \frac{3}{2} & 1 \\
1 & 1 & 2
\end{pmatrix}.
\]
\end{example}

\paragraph*{Matrix representation of up persistent Laplacian.}
In order to write the matrix representation of up persistent Laplacian we need to choose bases $\mathcal{B}_1$ and $\mathcal{B}_2$ for $f_q(\ker(\partial_q^K))$ and $f_q(\ker(\partial_q^K))^\perp \subseteq \Ima(f_q)$ respectively, where $f_q(\ker(\partial_q^K))^\perp$ denotes the orthogonal complement of $f_q(\ker(\partial_q^K))$ inside the ambient space $\Ima(f_q)$.
Let $\mathcal{D} = \{[\tau_{n+1}],...,[\tau_{n+l}] \} = \mathcal{S}_q^L - f_q(\pm \mathcal{S}_q^K)$. 
Then, $\mathcal{B}_1 \cup \mathcal{B}_2 \cup \mathcal{D}$ is basis for $C_q(L)$. 
Writing the coordinates of this new basis elements with respect to the canonical basis $\mathcal{S}_q^L$ as column vectors, we obtain the change of basis matrix 
\[M_{\mathcal{B}_1 \cup \mathcal{B}_2 \cup \mathcal{D} \to \mathcal{S}_q^L} = \begin{pmatrix}
R_1 & R_2 & 0_{n\times l} \\
0_{l\times \mathrm{rk}R_1}   & 0_{l\times \mathrm{rk}R_2}   & \mathbb{I}_l
\end{pmatrix}\]
where $R :=\begin{pmatrix}
R_1 & R_2
\end{pmatrix}$ is the $n\times n$ change of basis matrix from $\mathcal{B}_1 \cup \mathcal{B}_2$ to the canonical basis of $\Ima(f_q)$, and $\mathbb{I}_l$ is the $l \times l$ identity matrix.

Let $[\Delta_{q,\mathrm{up}}^{L}]$ be the matrix representation of $\Delta_{q,\mathrm{up}}^L$ with respect to the canonical basis of $C_q(L)$. Then, $Q = (M_{\mathcal{B}_1 \cup \mathcal{B}_2 \cup \mathcal{D} \to \mathcal{S}_q^L})^{-1} [\Delta_{q,\mathrm{up}}^{L}] M_{\mathcal{B}_1 \cup \mathcal{B}_2 \cup \mathcal{D} \to \mathcal{S}_q^L}$ is the matrix representation of $\Delta_{q,\mathrm{up}}^{L}$ with respect to $\mathcal{B}_1 \cup \mathcal{B}_2 \cup \mathcal{D} $. Let $n_p = \dim(f_q(\ker(\partial_q^K)))$ and let $E = Q([n_p], [n_p])$.Thus we can write $Q$ as a block matrix
\[Q = \begin{pmatrix}
E & F \\ 
G & H
\end{pmatrix}\]
where $F, G, H$ are chosen appropriately to $E$. We are now ready to write the matrix representation of $\Delta_{q,\mathrm{up}}^{K,L}$ with respect to the canonical basis of $\Ima(f_q)$.

\begin{proposition}\label{prop:upMatrix}
With the notations above, the matrix representation of $\Delta_{q,\mathrm{up}}^{K,L}$ with respect to the canonical basis of $\Ima(f_q)$ is given by
\begin{equation}\label{eq:upLapMat}
  \begin{pmatrix}
R_1 & R_2
\end{pmatrix}
\begin{pmatrix}
E - F H^\dagger G & 0_{n_p \times (n-n_p)}\\
0_{(n-n_p)\times n_p}                 & 0_{(n-n_p)\times (n-n_p)}
\end{pmatrix}
\begin{pmatrix}
R_1 & R_2
\end{pmatrix}^{-1}.  
\end{equation}
\end{proposition}

\begin{example}\label{ex:matrix rep up}
We will compute the matrix representation of the $1$st up persistent Laplacian of the weight preserving simplicial map depicted in~\autoref{fig:example}. We will stick to the notation used above. We start by choosing bases $\mathcal{B}_1$ and $\mathcal{B}_2$ for $f_q(\ker(\partial_q^K))$ and $f_q(\ker(\partial_q^K))^\perp \subseteq \Ima(f_q)$ respectively. Observe that $f_q(\ker(\partial_q^K))$ is spanned by $[xy]+[yz]-[xz]$. So, we can choose $\mathcal{B}_1 = \{ [xy]+[yz]-[xz]\}$ and $\mathcal{B}_2 = \{ 2[xy] - [yz], [yz]+[xz] \}$. As $f_1 : C_q^K \to C_1^L$ is surjective, we see that $\mathcal{D} = \emptyset$. Thus, $\mathcal{B}_1 \cup \mathcal{B}_2$ is a basis for $C_1^L$. Then, we have the change of basis matrix as
\[
M_{\mathcal{B}_1 \cup \mathcal{B}_2 \cup \mathcal{D} \to \mathcal{S}_q^L} = M_{\mathcal{B}_1 \cup \mathcal{B}_2  \to \mathcal{S}_q^L} = 
\begin{pmatrix}
1 & 2 & 0 \\
1 & -1 & 1 \\
-1 & 0 & 1
\end{pmatrix}
\]
where $\mathcal{S}_1^L = \{[xy], [yz], [xz] \}$ is the canonical (ordered) basis of $C_1^L$. Moreover, we get that $\begin{pmatrix}
R_1 & R_2
\end{pmatrix} = M_{\mathcal{B}_1 \cup \mathcal{B}_2} = M_{\mathcal{B}_1 \cup \mathcal{B}_2 \cup \mathcal{D} \to \mathcal{S}_q^L}$. With respect to $\mathcal{S}_1^L$, the matrix representation of $1$st combinatorial up Laplacian of $L$ is given by
\[
[\Delta_{1,\mathrm{up}}^{K,L}] = 
\begin{pmatrix}
\frac{1}{2} & 1 & -1 \\
\frac{1}{2} & 1 & -1 \\
-\frac{1}{2} & -1 & 1
\end{pmatrix}.
\]
Now, we compute 
\[
Q = (M_{\mathcal{B}_1 \cup \mathcal{B}_2 \cup \mathcal{D} \to \mathcal{S}_q^L})^{-1} [\Delta_{q,\mathrm{up}}^{L}] M_{\mathcal{B}_1 \cup \mathcal{B}_2 \cup \mathcal{D} \to \mathcal{S}_q^L} = 
\begin{pmatrix}
\frac{5}{2} & 0 & 0 \\
0 & 0 & 0 \\
0 & 0 & 0
\end{pmatrix}
\]
and, we extract $E = \begin{pmatrix}
\frac{5}{2}
\end{pmatrix}$, $F = \begin{pmatrix}
0 & 0
\end{pmatrix}$, $G=\begin{pmatrix}
0 \\
0
\end{pmatrix}$ and $H = \begin{pmatrix}
0 & 0\\
0 & 0
\end{pmatrix}$.
Thus, $E-FH^\dagger G = \begin{pmatrix}
\frac{5}{2}
\end{pmatrix}$. Thus, the matrix representation of $\Delta_{1,\mathrm{up}}^{K,L}$ with respect to the basis $\mathcal{S}_1^L = \{[xy], [yz], [xz] \}$ is given by
\[
[\Delta_{1,\mathrm{up}}^{K,L}]=
\begin{pmatrix}
R_1 & R_2
\end{pmatrix} \begin{pmatrix}
E - FH^\dagger G & 0 & 0\\
0 & 0 & 0 \\
0 & 0 & 0
\end{pmatrix} \begin{pmatrix}
R_1 & R_2
\end{pmatrix}^{-1} = \begin{pmatrix}
\frac{1}{2} & 1 & -1 \\
\frac{1}{2} & 1 & -1 \\
-\frac{1}{2} & -1 & 1
\end{pmatrix}.
\]
\end{example}

\begin{remark}
By combining~\autoref{ex:matrix rep down} and~\autoref{ex:matrix rep up}, we can see that the matrix representation of the $1$st persistent Laplacian $\Delta_1^{K,L}$ is given by
\[
[\Delta_1^{K,L}] = [\Delta_{1,\mathrm{down}}^{K,L}] + [\Delta_{1,\mathrm{up}}^{K,L}] = 
\begin{pmatrix}
\frac{7}{2} & 0 & 1 \\
0 & \frac{5}{2} & 0 \\
1/2 & 0 & 3
\end{pmatrix}.
\]
Then, we can justify~\autoref{persBetti} by observing that $\det([\Delta_1^{K,L}]) = 25 \neq 0$. That is, $\dim(\ker(\Delta_1^{K,L})) = 0 = \beta_1^{K,L}$.
\end{remark}

\begin{figure}[htb!]
    \centering
    \includegraphics{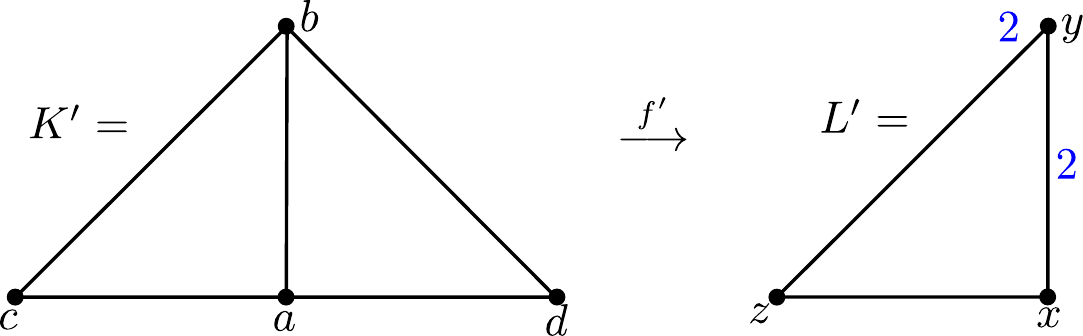}
    \caption{A weight preserving simplicial map $f' : K'\to L'$ between two weighted simplicial complexes $K'$ and $L'$. $K'$ has all the weights equal to 1. In $L'$, the edge $xy$ and the vertex $y$ has weights $2$ and the rest of the simplicies have weight $1$. The map $f'$ is given by $a\mapsto x$, $b\mapsto y$, $c\mapsto z$, $d\mapsto b$. And, ordering on the vertices are given by $a<b<c<d$ and $x<y<z$} 
    \label{fig:ex2}
\end{figure}

\begin{example}
Computing the matrix representation of the $1$st persistent Laplacian of the map $f' : K' \to L'$ depicted in~\autoref{fig:ex2} is similar to what we did for $f:K\rightarrow L$ in~\autoref{ex:matrix rep down} and~\autoref{ex:matrix rep up}. Actually, $[\Delta_{1,\mathrm{down}}^{K', L'}] = [\Delta_{1,\mathrm{down}}^{K, L}]$ as $C_1^K = C_1^{K'}$, $C_1^L = C_1^{L'}$, and $f_1=f_1'$ 
And, $[\Delta_{1,\mathrm{up}}^{K', L'}] = 0_{3\times 3}$ as $C_2^L = \{ 0 \}$. Thus,
\[
[\Delta_{1}^{K', L'}] =[\Delta_{1,\mathrm{down}}^{K', L'}] + [\Delta_{1,\mathrm{up}}^{K', L'}] = [\Delta_{1,\mathrm{down}}^{K, L}] + 0_{3\times 3} = \begin{pmatrix}
3 & -1 & 2\\
-\frac{1}{2} & \frac{3}{2} & 1 \\
1 & 1 & 2
\end{pmatrix}.
\]
Then, observe that $\dim(\ker(\Delta_1^{K',L'})) = 1 = \beta_1^{K',L'}$. Actually, the kernel of the matrix $[\Delta_1^{K',L'}]$ is generated by the vector $\begin{pmatrix}
1 & 1 & -1
\end{pmatrix}^\mathrm{T}$, which corresponds to the cycle $[xy]+[yz]-[xz]$ that can be seen as the image of the homology class that persists through the map $f'$.
\end{example}

\subsection{An Algorithm for Computing the Persistent Laplacian}
By~\autoref{prop:downMatrix} and~\autoref{prop:upMatrix}, we have the matrix representations of up and down persistent Laplacians with respect to the canonical basis of $\Ima{(f_q)}$. So, simply adding them up, gives us the matrix representation of the persistent Laplacian $\Delta_q^{K,L}$ with respect to the canonical basis. In the process for finding these matrices, we use explicit bases $S_q^K$, $S_q^L$, $\mathcal{B}\cup \mathcal{J}$ and $\mathcal{B}_1 \cup \mathcal{B}_2 \cup \mathcal{D}$. However, we do not have an explicit basis for $f_q(\ker(\partial_q^K))$. Yet, we do not need to compute $\ker(\partial_q^K)$ in order to compute $f_q(\ker(\partial_q^K))$ by the following lemma. 

\begin{lemma}\label{lem:fq ker = ker lap}
$f_q(\ker(\partial_q^K)) = \ker(\Delta_{q,\mathrm{down}}^{K,L})$.
\end{lemma}

 \begin{algorithm}[htb!]
 \caption{An algorithm for matrix representation of persistent Laplacian }
 \label{algo:matrix}
 \begin{algorithmic}[1]
 \STATE \textbf{Data:} $M_{\mathcal{J}\cup \mathcal{B} \to S_q^K}, [\Delta_{q,\mathrm{down}}^K], [\Delta_{q,\mathrm{up}}^L]$ and $W_{\Ima(f_q)}$
 \STATE \textbf{Result:} $[\Delta_q^{K,L}]$
 \STATE $N := M_{\mathcal{J}\cup \mathcal{B} \to S_q^K}^{-1} [\Delta_{q,\mathrm{down}}^{K}] M_{\mathcal{J}\cup \mathcal{B} \to S_q^K}$
 \STATE $n := \dim(W_{\Ima(f_q)})$
 \STATE $[\Delta_{q,\mathrm{down}}^{K,L}] := W_{\Ima(f_q)} (N / N([n_q^K] - [n], [n_q^K]-[n])) W_{\Ima(f_q)}^{-1}$
 \STATE Form $R = \begin{pmatrix}R_1 & R_2\end{pmatrix}$ by computing $\ker([\Delta_{q,\mathrm{down}}^{K,L}])$
 \STATE Expand  matrix $R$ with the identity matrix to form  $(n_q^L \times n_q^L)$ matrix $M_{\mathcal{B}_1 \cup \mathcal{B}_2 \cup \mathcal{D} \to S_q^L}$
 \STATE $Q:= M_{\mathcal{B}_1 \cup \mathcal{B}_2 \cup \mathcal{D} \to S_q^L}^{-1} [\Delta_{q,\mathrm{down}}^L] M_{\mathcal{B}_1 \cup \mathcal{B}_2 \cup \mathcal{D} \to S_q^L}$
 \STATE $n_p :=$ the number of columns of $R_1$
 \STATE $\mathsf{SchQ}:= Q/Q([n_q^L]-[n_p],[n_q^L]-[n_p])$
 \STATE Form the $n\times n$ matrix $\mathsf{PadSchQ}$ by zero padding to $\mathsf{SchQ}$
 \STATE $[\Delta_{q,\mathrm{up}}^{K,L}] = R^{-1} \mathsf{PadSchQ} \,R$
 \STATE{\RETURN $[\Delta_{q,\mathrm{down}}^{K,L}] + [\Delta_{q,\mathrm{up}}^{K,L}]$}
\end{algorithmic}
\end{algorithm}

\subsubsection{Complexity}
With the data we started in the~\autoref{algo:matrix}, we multiply matrices of dimension $n_q^K$ and take Schur complement in a matrix of dimension $n_q^K$ in order to compute $[\Delta_{q,\mathrm{down}}^{K,L}]$. Thus, it takes $O((n_q^K)^3)$ time to compute $[\Delta_{q,\mathrm{down}}^{K,L}]$. To compute $[\Delta_{q,\mathrm{up}}^{K,L}]$, we compute kernel of a matrix of dimension $n<n_q^L$, take Schur complement in a matrix of dimension $n_q^L$, multiply matrices of dimension $n_q^L$ and of dimension $n$. Hence, it takes $O((n_q^L)^3)$ time to compute $[\Delta_{q,\mathrm{up}}^{K,L}]$. Therefore, it takes $O((n_q^K)^3) + (n_q^L)^3)$ time to compute $[\Delta_q^{K,L}]$ in total.

It is important to note that the data we started in the \autoref{algo:matrix} also takes time to compute. Starting with boundary matrices and weight matrices, it takes $O((n_q^K)^2)$ time to compute $[\Delta_{q,\mathrm{down}}^{K}]$ and it takes $O(n_{q+1}^L)$ to compute $[\Delta_{q,\mathrm{up}}^L]$ as discussed in~\cite{memoli2022persistent}. Thus, starting from scratch, \autoref{algo:matrix} computes $[\Delta_q^{K,L}]$ in $O((n_q^K)^3 + (n_q^L)^3 + n_{q+1}^L)$ time.

Note that by Theorem \ref{persBetti}, as a by-product, the above algorithm can also output the persistent Betti number for a simplicial map $f: K \to L$ in the same time complexity. This provides an alternative way to compute persistent Betti numbers for $f: K \to L$ that is different from the existing algorithm by Dey et al. \cite{dey2012computing} already in the literature.

\section{Monotonicity of (up/down) persistent eigenvalues}\label{sec:mono}

For a simplicial map $f:K\to L$, the up and down persistent Laplacians are self-adjoint positive semi-definite operators. Therefore, they have non-negative eigenvalues. We denote them by $0 \leq \lambda_{q,\mathrm{up},1}^{K,L} \leq \lambda_{q,\mathrm{up},2}^{K,L} \ldots \leq \lambda_{q,\mathrm{up},n}^{K,L}$, and $0 \leq \lambda_{q,\mathrm{down},1}^{K,L} \leq \lambda_{q,\mathrm{down},2}^{K,L} \ldots \leq \lambda_{q,\mathrm{down},n}^{K,L}$, allowing repetition, where $n = \dim(\Ima{(f_q)})$. And, we call them the \emph{up persistent eigenvalues} and the \emph{down persistent eigenvalues}.

When the simplicial maps involved are inclusions, we have the following known monotonicity result for the up persistent Laplacian.

\begin{theorem}[{\cite[Theorem 5.3]{memoli2022persistent}}]\label{thm:up-mono-inc}
Let $f : K\hookrightarrow L$ and $g : L \hookrightarrow M$ be inclusion maps for simplicial complexes $K,L$ and $M$.
Then, for any $q\in\mathbb{N}$ and $k=1,2,\ldots, n_q^K$,
\[\lambda_{q,\mathrm{up}, k}^{K,M} \geq \lambda_{q,\mathrm{up}, k}^{L,M} \text{ and }\lambda_{q,\mathrm{up}, k}^{K,M} \geq \lambda_{q,\mathrm{up}, k}^{K,L}.\]
\end{theorem}

In~\autoref{thm:up-mono-inc}, the monotonicity result of up persistent eigenvalues $\lambda_{q,\mathrm{up}, k}^{K,M} \geq \lambda_{q,\mathrm{up}, k}^{K,L}$ follows from the fact that $\Delta_{q,\mathrm{up}}^{K,M} \succeq \Delta_{q,\mathrm{up}}^{K,L}$. In the case of surjective maps, we present an analogous statement for the down persistent Laplacians as follows.

\begin{proposition}\label{prop:surj-mon}
Let $f: K \twoheadrightarrow{L}$ and $g: L \twoheadrightarrow M$  be weight preserving surjective simplicial maps. Then, $\Delta_{q,\mathrm{down}}^{K,M} \succeq \Delta_{q,\mathrm{down}}^{L,M}$.
\end{proposition}

When the surjectivity assumption is removed, it is no longer guaranteed that the composition of two weight preserving maps is weight preserving, see~\autoref{ex:compositionNotWP}. However, under the assumption that two maps and their composition are weight preserving, we get the monotonicity of the down persistent eigenvalues.

\begin{theorem}\label{thm:down-mono}
Let $f : K\to L$ and $g : L \to M$ be weight preserving simplicial maps and assume that $g\circ f:K\rightarrow M$ is also weight preserving. Then, for any $q\in\mathbb{N}$ and $k=1,2,\ldots,\dim(\Ima{(g_q \circ f_q)})$, 
\[\lambda_{q,\mathrm{down}, k}^{K,M} \geq \lambda_{q,\mathrm{down}, k}^{L,M} \text{ and }\lambda_{q,\mathrm{down}, k}^{K,M} \geq \lambda_{q,\mathrm{down}, k}^{K,L}.\]
\end{theorem}

However, this type of monotonicity does not hold in general for up persistent eigenvalues even if we require weight preserving conditions for the involved simplicial maps as we did in \autoref{thm:down-mono}. See the counterexample as follows.

\begin{figure}[htb!]
    \centering
    \includegraphics{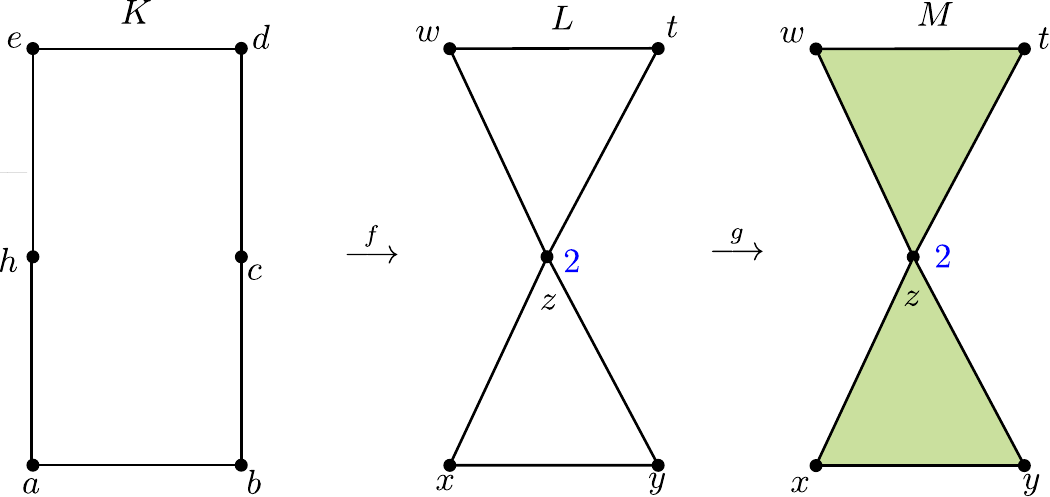}
    \caption{Composition of two weight preserving simplicial maps $f: K\to L$ and $g: L\to M$, where $f$ is given by collapsing the vertices $h$ and $c$ to the same vertex $z$. And, $g$ is given by the identity map on the vertices.} 
    \label{fig:ex3}
\end{figure}

\begin{example}[Up persistent eigenvalues are not monotonic]\label{ex:not monotonic}
Considering the simplicial complexes $K$, $L$, $M$ and the simplicial maps $f$, $g$ depicted in~\autoref{fig:ex3}, we compute spectra of $\Delta_{1,\mathrm{up}}^{K,M}$ and $\Delta_{1,\mathrm{up}}^{L,M}$. It turns out that $\Delta_{1,\mathrm{up}}^{K,M}$ has eigenvalues $0 \leq 0 \leq 0\leq 0\leq 0\leq 3$ and $\Delta_{1,\mathrm{up}}^{L,M}$ has eigenvalues $0 \leq 0 \leq 0\leq 0\leq 3\leq 3$. So, $0=\lambda_{1,\mathrm{up},5}^{K,M} \ngeq \lambda_{1,\mathrm{up},5}^{L,M} =3$.

\end{example}

Recall from \autoref{thm:schur-persistent} that $\Delta_{q,\mathrm{up}}^{K,L} = \iota_{\Ima{(f_q)}} \circ \Sch(\Delta_{q,  \mathrm{up}}^L, f_q(\ker(\partial_q^K))) \circ \proj_{f_q(\ker(\partial_q^K))}$.
This formulation reveals that the up persistent Laplacian is obtained by extending the operator $\Sch(\Delta_{q,  \mathrm{up}}^L, f_q(\ker(\partial_q^K)))$ defined on $f_q(\ker(\partial_q^K))$ to its superspace $\Ima{(f_q)}$ by ``padding zeros''.
This extension naturally introduces inevitable $0$ eigenvalues to the up persistent Laplacian and we call them inevitable $0$ eigenvalues.
Considering again \autoref{ex:not monotonic}, we see that $g_1(f_1(\ker(\partial_1^K)))$ has dimension $1$ and codimension $5$ inside $\Ima{(g_1\circ f_1)}$. 
Thus, $\Delta_{1,\mathrm{up}}^{K,M} $ has $5$ inevitable $0$ eigenvalues. Similarly, $\Delta_{1,\mathrm{up}}^{L,M} $ has $4$ inevitable $0$ eigenvalues as the codimension of $g_1(\ker(\partial_1^L))$ inside $\Ima{(g_1)}$ is $4$. Disregarding these inevitable $0$ eigenvalues from their spectra, we see that $\Delta_{1\mathrm{up}}^{K,M}$ essentially has $\{3 \}$ as its spectrum, while $\Delta_{1,\mathrm{up}}^{L,M}$ essentially has $\{3, 3\}$ as its spectrum. Then, it seems that if we disregard inevitable $0$ eigenvalues, we will obtain monotonicity for the eigenvalues of up persistent Laplacians. This is indeed the case:

We call $\Sch(\Delta_{q,  \mathrm{up}}^L, f_q(\ker(\partial_q^K)))$ the \emph{essential up persistent Laplacian}, whose spectrum is the same as the spectrum of $\Delta_{q,\mathrm{up}}^{K,L}$ up to a difference in the multiplicity of the 0 eigenvalue.
Then, we establish monotonicity of the eigenvalues of the essential up persistent Laplacian, which are denoted by $\lambda_{q,\mathrm{up},k}^{K,L,\mathrm{ess}}$, and are called \emph{essential up persistent eigenvalues}.

\begin{theorem}\label{thm:mono-up-ess}
Let $f: K\to L$ and $g: L\to M$ be weight preserving simplicial maps. Then,  for any $q\in\mathbb{N}$ and $k=1,2,\ldots, \dim(g_q(f_q(\ker(\partial_q^K))))$, we have $\lambda_{q,\mathrm{up},k}^{K,M,\mathrm{ess}}\geq \lambda_{q,\mathrm{up},k}^{L,M,\mathrm{ess}}$.
\end{theorem}

This monotonicity result on essential up persistent eigenvalues is stronger than the monotonicity result for inclusion maps (cf.~\autoref{thm:up-mono-inc}) in that the latter is a direct consequence of the former.

\section{Discussion}

Once an invariant is associated to a simplicial filtration/tower, one of the most natural questions would be about its stability. So, it is highly desirable to explore the stability of the (up/down) persistent eigenvalues/eigenspaces that could potentially generalize the stability of up persistent eigenvalues in the inclusion-based persistent Laplacian~\cite[Theorem 5.10]{memoli2022persistent}.

The persistent diagram of a Rips complex can be approximated by using simplicial towers obtained from the Rips complex such as sparsified Rips complex or graph induced complex as described in~\cite{dey2013graph, dey2012computing}. Therefore, one might consider if the spectrum of the (up/down) persistent Laplacian can also be approximated via a similar sparsification process.

\bibliography{ref}

\appendix

\section{Details for Section \ref{sec:basics-and-defns}}
\begin{proof}[Proof of Lemma \ref{lem:isometry}]
Let $T_q^f$ be the subspace of $C_q^K$ generated by $\mathcal{J}$. Observe that $\ker(f_q)$ is generated by 
\[\{ \sgn_{f_q}(\sigma) [\sigma] - \sgn_{f_q}(\sigma') [\sigma'] \mid f_q([\sigma]) = f_q([\sigma']) \neq 0\} \cup \{ [\sigma] \in S_q^K \mid f_q([\sigma]) = 0 \}.\]
Therefore, if $[\tau] \in \Ima(f_q)$ is a $q$-simplex in $L$ and $f_q([\sigma]) = f_q([\sigma']) =[\tau]$ for some $\sigma, \sigma' \in S_q^K$, then
\[
\langle c_q^\tau,  \sgn_{f_q}(\sigma) [\sigma] - \sgn_{f_q}(\sigma')[\sigma'] \rangle = w_q^K(\sigma) \langle [\sigma], [\sigma] \rangle - w_q^K (\sigma') \langle  [\sigma'], [\sigma'] \rangle = 1-1 = 0.
\]
Similarly, we get $0$ as the result of inner product of $c_q^\tau$ with other generators of the space $\ker(f_q)$. It then follows that $T_q^f \subseteq \ker(f_q)^\perp$. 
\begin{observation}\label{obs:orthogonal-basis}
$\mathcal{J} = \{ c_q^{\tau}  \mid  [\tau]\in \Ima(f_q) \}$ is an orthogonal basis for $T_q^f$, and $f_q(\mathcal{J}) = \{f_q(c_q^\tau) \mid [\tau]\in \Ima(f_q)\}$ is an orthogonal basis for $\Ima(f_q)$
\end{observation}

\begin{proof}
Let $[\tau_1], [\tau_2] \in \Ima(f_q)$. If $f_q([\sigma_1]) = [\tau_1] \neq [\tau_2] = f_q([\sigma_2])$, then $\sigma_1 \neq \sigma_2$. Thus, $\langle [\sigma_1], [\sigma_2] \rangle_{w_q^K} =0$. Thus, $\langle c_q^{\tau_1}, c_q^{\tau_2} \rangle_{w_q^K} = 0$. Thus, $\{ c_q^{\tau}  \mid  [\tau]\in \Ima(f_q) \}$ is an orthogonal basis for $T_q^f$. And we have that $f_q(c_q^{\tau}) = w_q^L(\tau)[\tau]$. Thus, $\{f_q(c_q^\tau) \mid [\tau]\in \Ima(f_q)\}$ is a scaling of the canonical basis of $\Ima(f_q)$. Therefore, it is a orthogonal basis for $\Ima(f_q)$
\end{proof}
By the~\autoref{obs:orthogonal-basis}, we have that $f_q$ restricted to $T_q^f$ is a linear isomorphism $f_q|_{T_q^f} : T_q^f \to \Ima(f_q)$. This implies that $\dim(T_q^f) = \dim(\Ima{(f_q)}) = \dim(\ker(f_q)^\perp)$. Hence we have that $T_q^f = \ker(f_q)^\perp$. Thus $\mathcal{J}$ is an orthogonal basis for $\ker(f_q)^\perp$.

As the set $\mathcal{J} = \{ c_q^\tau \mid \tau \in \Ima(f_q) \}$ is an orthogonal basis for $T_q^f = \ker(f_q)^\perp$ and $\{f_q(c_q^\tau) \mid \tau\in \Ima(f_q)\}$ is an orthogonal basis for $\Ima(f_q)$, in order to see that $\hat{f}_q$ preserves inner product, it is enough that check that $\langle c_q^\tau , c_q^{\tau} \rangle_{w_q^K} = \langle f_q(c_q^\tau), f_q(c_q^{\tau}) \rangle_{w_q^L}$ for every $[\tau] \in \Ima(f_q)$. Let $c_q^\tau = \sum_{i=1}^l \sgn_{f_q}(\sigma_i) w_q^K(\sigma_i) [\sigma_i]$ for some $[\tau] \in \Ima(f_q)$. By the assumption that $f$ is weight preserving, we have that $\sum_{i=1}^l w_q^K(\sigma_i) = w_q^L(\tau)$ and $f_q(c_q^\tau) = w_q^L(\tau) [\tau]$. Then, it follows that
\[
        \langle c_q^\tau , c_q^{\tau} \rangle_{w_q^K} = \sum_{i=1}^l w_q^K(\sigma_i) = w_q^L(\tau) = \langle w_q^L(\tau) [\tau], w_q^L(\tau) [\tau] \rangle_{w_q^L} = \langle f_q(c_q^\tau), f_q(c_q^\tau) \rangle_{w_q^L}.
\]
This completes the proof. 
\end{proof}

\subsection{A cochain formulation of the persistent Laplacian}\label{sec:cochain laplacian}

Recall from \autoref{rmk:cochain laplacian} that
\[\mathfrak{C}_{L\leftarrow K}^{q+1} := \{ c\in C^{q+1}_L \mid (\delta_{q}^L)^*(c) \in (f^q)^*(\ker((\delta_{q-1}^K)^*))\},\]
\[\mathfrak{C}_{K\rightarrow L}^{q-1} := \{c\in C^{q-1}_K \mid \delta^K_{q-1}(c) \in \ker((f^q)^*)^\perp \}.\]

Similarly as in \autoref{lem:isometry},
the restriction of $f^q$ onto $\ker(f^q)^\perp$ gives rise to an isometry $\hat{f}_q : \ker(f^q)^\perp \to \Ima(f^q)$.

Let $\delta^{q-1}_{K,L}$ denote the restriction of $\delta^{q-1}_K$ to $\mathfrak{C}_{K\rightarrow L}^{q-1}$. Let $\partial^{q+1}_{L,K}$ denote the restriction of $(\delta_{q}^L)^*$ to $\mathfrak{C}_{L\leftarrow K}^{q+1} $.

We then draw the following diagram dual to the one on page \pageref{tik:laplacian}.

\begin{center}
\begin{tikzcd}
                                                                                &  & C^q_K                                                                                                                   &  & C^{q-1}_K                                                                        \\
                                                                                &  & \Ima(f^q) \arrow[u, hook] \arrow[dd, "(\hat{f}^q)^{-1}"', bend right] \arrow[rr, red, "{(\delta^{q-1}_{K,L})^*}", bend left]          &  & {\mathfrak{C}_{K\rightarrow L}^{q-1}} \arrow[u, hook] \arrow[ll, red, "{\delta^{q-1}_{K,L}}"', bend left] \\
                                                                                &  &                                                                                                                         &  &                                                                                  \\
{\mathfrak{C}_{L\leftarrow K}^{q+1}} \arrow[rr, blue, "{\partial^{q+1}_{L,K}}" ', bend left] \arrow[d, hook] &  & \ker(f^q)^\perp \arrow[uu, "\hat{f}^q"', bend right] \arrow[ll, blue, "{(\partial^{q+1}_{L,K})^*}", bend left] \arrow[d, hook] &  &                                                                                  \\
C^{q+1}_L                                                                       &  & C^q_L                                                                                                                   &  &                                                                                 
\end{tikzcd}
\end{center}

Then, based on this diagram, we define as follows certain operators like our (up/down) persistent Laplacians defined in \autoref{sec:perslap for simplicial}.
\begin{align*}
    \Delta^{q,\mathrm{up}}_{K\stackrel{f}{\to} L} :=& \partial^{q+1}_{L,K}\circ (\partial^{q+1}_{L,K})^*:\ker(f^q)^\perp\rightarrow \ker(f^q)^\perp, \\
    \Delta^{q,\mathrm{down}}_{K\stackrel{f}{\to} L} :=& (\hat{f}^q)^{-1} \circ \delta^{q-1}_{K,L} \circ (\delta^{q-1}_{K,L})^* \circ \hat{f}^q:\ker(f^q)^\perp\rightarrow \ker(f^q)^\perp, \\
    \Delta^{q}_{K\stackrel{f}{\to} L}:=&\Delta^{q,\mathrm{up}}_{K\stackrel{f}{\to} L}+\Delta^{q,\mathrm{down}}_{K\stackrel{f}{\to} L}:\ker(f^q)^\perp\rightarrow \ker(f^q)^\perp.
\end{align*}

It turns out that the operators defined above are the same as (up/down) persistent Laplacians up to certain isometry. We specify this point more rigorously as follows.

Recall the isometry $j_q^K:C_q^K\rightarrow C^q_K$ between the chain group and the cochain group. 
\begin{lemma}\label{lm:im ker commute}
The restriction of $j_q^K$ onto $\ker(f_q)^\perp$ gives rise to an isometry $\hat{j}_q^K:\ker(f_q)^\perp\rightarrow\Ima(f^q)$. Similarly, the restriction of $j_q^L$ onto $\Ima(f_q)$ gives rise to an isometry $\hat{j}_q^L:\Ima(f_q)\rightarrow\ker(f^q)^\perp$. Moreover, the following diagram commutes
\begin{center}
\begin{tikzcd}

\Ima(f_q) \arrow[r, "(\hat{f}_q)^*"]\arrow[d,"\hat{j}_q^L"'] &  \ker(f_q)^\perp\arrow[d,"\hat{j}_{q}^K"]\\
\ker(f^q)^\perp\arrow[r, "\hat{f}^q"]  & \Ima(f^q)
  
\end{tikzcd}.
\end{center}
\end{lemma}
\begin{proof}
This simply follows from the fact that the following diagram commutes
\begin{center}
\begin{tikzcd}

C^K_q \arrow[d,"j_q^K"'] & \arrow[l, "(f_q)^*"'] C^L_{q}\arrow[d,"j_{q}^L"]\\
C^q_K & \arrow[l, "f^q"'] C^{q}_L
  
\end{tikzcd}.
\end{center}
and that $\ker(f_q)^\perp=\Ima((f_q)^*)$ and $\Ima(f_q)=\ker((f_q)^*)^\perp$.
\end{proof}

Then, we have the following result which basically states that (up/down) persistent Laplacians can be constructed either via chains or via cochains and the two types of constructions are dual with each other.
\begin{theorem}
 For any $q\in\mathbb{N}$, we have that
 \[\hat{j}_q^L\circ\Delta_{q,\mathrm{up}}^{K\stackrel{f}{\to} L}= \Delta^{q,\mathrm{up}}_{K\stackrel{f}{\to} L}\circ \hat{j}_q^L,\,\,\hat{j}_q^L\circ\Delta_{q,\mathrm{down}}^{K\stackrel{f}{\to} L}= \Delta^{q,\mathrm{down}}_{K\stackrel{f}{\to} L}\circ \hat{j}_q^L\,\,\text{and }\hat{j}_q^L\circ\Delta_{q}^{K\stackrel{f}{\to} L}= \Delta^{q}_{K\stackrel{f}{\to} L}\circ \hat{j}_q^L\]
\end{theorem}
\begin{proof}
We note that the following diagram commutes with all vertical arrows being isometries.
\begin{center}
\begin{tikzcd}

C^L_{q+1} \arrow[r, "\partial_{q+1}^L"]\arrow[d,"j_{q+1}^L"']&C^L_q \arrow[d,"{j}_q^L"']& C^K_q \arrow[l, "f_q"']\arrow[d,"{j}_q^K"'] &  C^K_{q-1} \arrow[l, "(\partial_{q-1}^K)^*"']\arrow[d,"{j}_{q-1}^K"']\\
C_L^{q+1} \arrow[r, "(\delta^{q}_L)^*"']&C_L^q & C_K^q \arrow[l, "({f}^q)^*"]&  C_K^{q-1} \arrow[l, "\delta^{q-1}_K"]
  
\end{tikzcd}.
\end{center}
By slight abuse of notation, this commutative diagram immediately gives rise to two isometries 
\[\hat{j}_{q-1}^K:\mathfrak{C}^{K\rightarrow L}_{q-1}\rightarrow\mathfrak{C}_{K\rightarrow L}^{q-1}\,\,\text{and}\,\,\hat{j}_{q+1}^L:\mathfrak{C}^{L\leftarrow K}_{q+1}\rightarrow\mathfrak{C}_{L\leftarrow K}^{q+1}.\]
Furthermore, the following diagram commutes:
\begin{center}
\begin{tikzcd}

\mathfrak{C}^{L\leftarrow K}_{q+1} \arrow[r, "\partial_{q+1}^{L,K}"]\arrow[d,"\hat{j}_{q+1}^L"']&\Ima(f_q) \arrow[d,"\hat{j}_q^L"']\arrow[r, "(\hat{f}_q)^*"]& \ker(f_q)^\perp \arrow[d,"\hat{j}_q^K"'] &  \mathfrak{C}^{K\rightarrow L}_{q-1} \arrow[l, "\delta_{q-1}^{K,L}"']\arrow[d,"\hat{j}_{q-1}^K"']\\
\mathfrak{C}_{L\leftarrow K}^{q+1} \arrow[r, "\partial_{L,K}^{q+1}"']&\ker(f^q)^\perp\arrow[r, "\hat{f}^q"']& \Ima(f^q) &  \mathfrak{C}_{K\rightarrow L}^{q-1} \arrow[l, "\delta^{q-1}_{K,L}"]
  
\end{tikzcd}.
\end{center}
Then, by taking adjoints of horizontal arrows of the diagram above, one still obtains a commutative diagram. Then, simply by following the definitions, we conclude the proof.
\end{proof}

\section{Details for Section \ref{sec:schur}}\label{app:secSchur}

\begin{lemma}\label{lem:cancellation}
Let $R\in \R^{n\times n}$ be an invertible matrix and $E \in \R^{n\times m}$ be any matrix. Then, 
\[
E^\mathrm{T} R (R^{-1} E E^\mathrm{T} R)^\dagger R^{-1} E = E^\mathrm{T} (E E^\mathrm{T})^\dagger E = E^\dagger E.
\]

\end{lemma}

\begin{proof}
To prove the equalities, we need the following fact
\begin{claim}[{~\cite[Proposition 3.3]{pseudoinverse}}]\label{claim:projector}
Let $H \in \R^{n\times m}$. Then, $HH^\dagger = \pi_{\Ima(H)} = \mathbb{I}_n - \pi_{\ker(H^\mathrm{T})}$, where $\mathbb{I}_n$ is the $n\times n$ identity matrix and for any subspace $W \subseteq \R^n$,  $\pi_W \in \R^{n\times n}$ is the orthogonal projector onto $W$. Similarly, $H^\dagger H = \pi_{\Ima(H^\mathrm{T})} = \mathbb{I}_m - \pi_{\ker(H)}$. 
\end{claim}
Using~\autoref{claim:projector}, we get
\begin{align*}
E^\mathrm{T} R (R^{-1} E E^\mathrm{T} R)^\dagger R^{-1} E =& \pi_{\Ima(E^\mathrm{T})} E^\mathrm{T} R (R^{-1} E E^\mathrm{T} R)^\dagger R^{-1} E \\
                                        =& E^\dagger E E^\mathrm{T} R (R^{-1} E E^\mathrm{T} R)^\dagger R^{-1} E \\
                                        =& E^\dagger R (R^{-1} E E^\mathrm{T} R ) (R^{-1} E E^\mathrm{T} R)^\dagger R^{-1} E \\
                                        =& E^\dagger R (\mathbb{I}_n - \pi_{\ker(R^\mathrm{T} E E^\mathrm{T} (R^{-1})^\mathrm{T})})R^{-1}E \\
                                        =& E^\dagger R (\mathbb{I}_n - \pi_{\ker(E^\mathrm{T} (R^{-1})^\mathrm{T})})R^{-1}E \\
                                        =& E^\dagger R (\mathbb{I}_n - \pi_{\ker( (R^{-1} E)^\mathrm{T})})R^{-1}E \\
                                        =& E^\dagger R (\mathbb{I}_n - \pi_{\Ima( (R^{-1} E))^\perp})R^{-1}E \\
                                        =& E^\dagger R R^{-1} E \\
                                        =& E^\dagger E.
\end{align*}
And, by taking $R = \mathbb{I}_n$, we also get $E^\mathrm{T} (E E^\mathrm{T})^\dagger E = E^\dagger E$.
\end{proof}

\begin{proof}[Proof of~\autoref{prop:well-defn}]
Let $\mathcal{B}_1$ and $\mathcal{B}_2$ be ordered bases for $W$ and $W^\perp$ respectively and let $n= \dim(V)$, $d = \dim(W)$. Writing $L$ with respect to the ordered basis $\mathcal{B}_1 \cup \mathcal{B}_2$ in which the order is given by extending the orders on $\mathcal{B}_1$ and $\mathcal{B}_2$ by asserting that $\mathcal{B}_1 < \mathcal{B}_2$, we get the block matrix representation 
\[
[L]_{\mathcal{B}_1\cup \mathcal{B}_2} = \begin{pmatrix}
A & B \\
C & D
\end{pmatrix}
\]
where $A$ is $d\times d$ and $D$ is $(n-d)\times (n-d)$ square matrices.

Let $\mathcal{C}_1$ and $\mathcal{C}_2$ be any orthonormal bases for $W$ and $W^\perp$ respectively.
Let $P$ and $R$ be the change of basis matrices from $\mathcal{B}_1$ to $\mathcal{C}_1$ and from $\mathcal{B}_2$ to $\mathcal{C}_2$ respectively. Since $[L]_{\mathcal{C}_1 \cup \mathcal{C}_2}$ is a positive semi-definite matrix, we have that
\[
    \begin{pmatrix}
    P^{-1} & 0 \\
    0 & R^{-1}
    \end{pmatrix}
    \begin{pmatrix}
    A & B \\
    C & D
    \end{pmatrix}
    \begin{pmatrix}
    P & 0 \\
    0 & R
    \end{pmatrix}
    = [L]_{\mathcal{C}_1 \cup \mathcal{C}_2} = EE^\mathrm{T}
\]
for some $E \in \R^{n\times n}$. Writing $E = \begin{pmatrix}
E_1 \\
E_2
\end{pmatrix}$ as a block matrix where $E_1 \in \R^{d\times n}$ and $E_2 \in \R^{(n-d)\times n}$. Thus, to show that the Schur restriction is independent of choice of basis, we need to show that:
\[
P^{-1}(A - B D^\dagger C) P = E_1 E_1^\mathrm{T} - E_1 E_2^\mathrm{T} (E_2 E_2^\mathrm{T})^\dagger E_2 E_1^\mathrm{T}.
\]

By computing the left-hand side, we get that
\begin{align*}
    P^{-1}(A - B D^\dagger C) P =& E_1 E_1^\mathrm{T} - E_1 E_2^\mathrm{T} R^{-1}(R E_2 E_2^\mathrm{T} R^{-1})^\dagger R E_2 E_1^\mathrm{T} \\
                                =& E_1 E_1^\mathrm{T} - E_1 E_2^\mathrm{T} (E_2 E_2^\mathrm{T})^\dagger E_2 E_1^\mathrm{T}
\end{align*}
where the last equality follows from~\autoref{lem:cancellation} and this finishes the proof.
\end{proof}

Now, we provide the following lemma that will be useful in the proof of~\autoref{prop:schur-restriction}.

\begin{lemma}\label{lemmaForMaximality}
 Let $V$ be a finite dimensional inner product space and let $L: V\to V$ be a self-adjoint positive semi-definite operator. Let $W \subseteq V$ be a subspace. Then, for every $w\in W$, there is an element $w^\perp \in W^\perp$ such that $L(w+w^\perp) \in W$.
\end{lemma}

\begin{proof}
Choose orthonormal bases for $W$, $\mathcal{B}_W = \{ w_1,...,w_d\}$, and for $W^\perp$, $\mathcal{B}_{W^\perp}\{ v_{d+1}, ..., v_{n} \}$, to write a matrix representation of $L$. Namely, $[L] = \begin{pmatrix}
A & B \\
B^\mathrm{T} & D
\end{pmatrix}$ with respect to the basis $\mathcal{B}_W \cup \mathcal{B}_{W^\perp}$. Let $w \in W$ and let $c = [c_1 . . . c_d]^\mathrm{T} \in \R^d$ be its coordinates, i.e. $w = [w_1 . . . w_d] [c_1 . . .c_d]^\mathrm{T} = \sum_{i=1}^d c_i w_i$. Choose coordinates for $w^\perp$ to be $-D^{\dagger} B^\mathrm{T} c \in \R^{n-d}$. Then, 
\[
[L]\begin{pmatrix}
c \\
-D^\dagger B^\mathrm{T} c
\end{pmatrix} = \begin{pmatrix}
(A - B D^\dagger B^\mathrm{T})c \\
(I - D D^\dagger)B^\mathrm{T} c
\end{pmatrix}
= \begin{pmatrix}
(A - B D^\dagger B^\mathrm{T})c \\
0
\end{pmatrix}
\]
by~\cite[Theorem 16.1]{geomethods}, which states that $(I - D D^\dagger)B^\mathrm{T} = 0$ when [L] is positive semi-definite . Thus, $L(w+w^\perp) \in W$.
\end{proof}

In order to prove~\autoref{prop:schur-restriction}, we will use the so-called extremal characterization of Schur complement, which is given by:

\begin{theorem}[\cite{extremalschur}, Extremal characterization of Schur complement]\label{exschur-matrix}
 Let $M = \begin{pmatrix}
 A & B \\
 B^\mathrm{T} & D
 \end{pmatrix}$ be a positive semi-definite real matrix. Then,
\[
A - BD^\dagger B^\mathrm{T} = \max \left\{ N: M- \begin{pmatrix}
N & 0_{d\times (n-d)} \\
0_{(n-d)\times d} & 0_{(n-d)\times (n-d)} 
\end{pmatrix} \succeq 0, N\text{ is } d\times d \text{ positive semi-definite}\right\}.
\]
\end{theorem}

This characterization is given by in terms of matrices. However we need it in terms operators. Combining the extremal characterization of Schur complement in matrices,~\autoref{exschur-matrix} , and the basis invariance of Schur restriction on operators,~\autoref{prop:well-defn} , we get the following result.

\begin{corollary}\label{exschur}
 Let $V$ be finite dimensional real inner product space and let $L: V \to V$ be a self-adjoint positive semi-definite operator. Let $W \subseteq V$ be a subspace. Then,
 \[
 \Sch(L, W) = \max \{ M: W \to W \text{self-adjoint positive semi-definite} \mid L\succeq \iota_V \circ M\circ \proj_W \}
 \]
 where $\iota_V : W\hookrightarrow V$ is the inclusion map and $\proj_W : V\to W$ is the projection map.
\end{corollary}

\begin{remark}\label{rem:adjoint}
Let $\mathrm{PSD}(V)$ denote the Loewner poset of the self-adjoint positive semi-definite operators on a finite dimensional real inner product space $V$. Consider a subspace $W\subseteq V$. Then, there is an order preserving map $\mathcal{E}: \mathrm{PSD}(W) \to \mathrm{PSD}(V)$ given by $M \mapsto \iota_V \circ M \circ \proj_W$, where $\iota_V : W \hookrightarrow V$ is the inclusion and $\proj_W : V\to W$ is the projection map. If we consider the poset categories $\mathrm{PSD}(W)$, and $\mathrm{PSD}(V)$,~\autoref{exschur} is essentially stating that $\Sch(-,W): \mathrm{PSD}(V) \to \mathrm{PSD}(W)$ is a right adjoint to $\mathcal{E}$.
\end{remark}

\begin{proof}[Proof of~\autoref{prop:schur-restriction}]
Let 
\[S = \{ M: W\to W \text{ self-adjoint positive semi-definite} \mid f\circ f^* - \iota_V \circ M\circ \proj_{W} \succeq 0\}.\]
By \autoref{exschur}, it is enough to show that $f_W \circ f_W^* = \max S$. Observe that 
\[
\underbrace{\iota_V \circ f_W}_{\hat{f}_W} \circ \underbrace{f_W^* \circ \proj_W}_{\hat{f}_W^*} = \hat{f}_W \circ \hat{f}_W^*,
\]
where $\hat{f}_W: f^{-1}(W) \to V$ is the restriction of $f$ on $f^{-1}(W)$.
Write $\hat{V} = f^{-1}(W) \oplus f^{-1}(W)^\perp$. Then, we have $f = \hat{f}_W \oplus \hat{f}_W^\perp$, where $\hat{f}_W^\perp : f^{-1}(W)^\perp \to V$. Thus, we get that $f\circ f^* = \hat{f}_W \circ \hat{f}_W^* + \hat{f}_W^\perp \circ (\hat{f}_W^\perp)^*$. 
Therefore, $f\circ f^* - \hat{f}_W \circ \hat{f}_W^* = \hat{f}_W^\perp \circ (\hat{f}_W^\perp)^*$ is positive semi-definite. As $\iota_V \circ f_W \circ f_W^* \circ \proj_W = \hat{f}_W \circ \hat{f}_W^*$, we get that $f\circ f^* - \iota_V \circ f_W \circ f_W^* \circ \proj_W = \hat{f}_W^\perp \circ (\hat{f}_W^\perp)^*$ is positive semi-definite. Thus, $f_W \circ f_W^* \in S$. 

To show maximality, we choose any $T \in S$, i.e., $f\circ f^* - \iota_V \circ T\circ \proj_W$ is positive semi-definite. The assumption that $f\circ f^* - \iota_V \circ T\circ \proj_W$ is positive semi-definite is equivalent to 
\begin{equation}
    \langle (f\circ f^* - \iota_V\circ T \circ \proj_W)(w+w^\perp), w+w^\perp \rangle \geq 0,\,\, \forall w\in W, w^\perp \in W^\perp.
\end{equation}
By bilinearity of the inner product, $\forall w\in W, w^\perp \in W^\perp$, the above inequality can be equivalently written as\
\begin{align}
    &\langle f\circ f^* (w+w^\perp), w+w^\perp \rangle \geq \langle T(w), w \rangle,\,\,\  \\
    \Longleftrightarrow&\langle f^* (w+w^\perp), f^*(w+w^\perp) \rangle \geq \langle T(w), w \rangle,\,\,\  \\
    \Longleftrightarrow&\langle f^*(w), f^*(w) \rangle + 2\langle f^*(w), f^*(w^\perp) \rangle + \langle f^*(w^\perp), f^*(w^\perp) \rangle \geq \langle T(w), w \rangle.\  \label{sum-with-adjoints}
\end{align}
Let $\mathcal{B} = \{ \hat{v}_1, ..., \hat{v}_l, \hat{v}_{l+1}, ..., \hat{v}_n\}$ be an orthonormal basis for $\hat{V}$ such that $\{\hat{v}_1 ,..., \hat{v}_l \}$ is a basis for $f^{-1}(W)$ and $\{ \hat{v}_{l+1}, ..., \hat{v}_n\}$ is a basis for $(f^{-1}(W))^\perp$. Then, we can write $f^*(v) = \sum_{i=1}^n \langle v, f(\hat{v}_i) \rangle \hat{v}_i$. Then, Equation~\eqref{sum-with-adjoints} becomes
\begin{equation}\label{assumption-psd}
    \sum_{i=1}^n \langle w, f(\hat{v}_i)\rangle ^2 + 2\sum_{i=l+1}^n \langle w, f(\hat{v_i})\rangle \langle w^\perp, f(\hat{v}_i)\rangle + \sum_{i=l+1}^n \langle w^\perp, f(\hat{v}_i)\rangle^2 
    \geq \langle T(w), w\rangle
\end{equation}
for all $w \in W$, $w^\perp \in W^\perp$. By~\autoref{lemmaForMaximality}, for each $w\in W$ we pick an element $w^\perp \in W^\perp$ such that $f\circ f^* (w+w^\perp) \in W$. Then, $f^* (w+w^\perp)\in f^{-1}(W)$ and thus Equation~\eqref{assumption-psd} becomes,
\begin{equation}
    \sum_{i=1}^l \langle w, f(\hat{v}_i)\rangle^2  \geq \langle Tw, w\rangle
\end{equation}
for all $w \in W$, which is equivalent to $f_W \circ f_W^* - T$ being positive semi-definite by reversing the procedure above. Thus, $f_W \circ f_W^* = \max S = \Sch(f\circ f^*, W)$.
\end{proof}

\begin{proof}[Proof of~\autoref{thm:schur-persistent}]
By~\autoref{prop:schur-restriction}, $\delta_{q-1}^{K,L} \circ (\delta_{q-1}^{K,L})^* = \Sch(\Delta_{q,\mathrm{down}}^K, \ker(f_q)^\perp)$. Thus, $\Delta_{q,\mathrm{down}}^{K,L} = \hat{f}_q \circ \delta_{q-1}^{K,L} \circ (\delta_{q-1}^{K,L})^* \circ \hat{f}_q^{-1} = \hat{f}_q \circ \Sch(\Delta_{q,\mathrm{down}}^K, \ker(f_q)^\perp) \circ \hat{f}_q^{-1}$.

Since $\Ima{ (\partial_{q+1}^{L,K})}\subseteq f_q(\ker(\partial_q^K))$, we let $\Bar{\partial}_{q+1}^{L,K} : \ChUp{L}{K} \to f_q(\ker(\partial_q^K)$ be given by the same formula as $\partial_{q+1}^{L,K}$.  By~\autoref{prop:schur-restriction}, $\Bar{\partial}_{q+1}^{L,K} \circ (\Bar{\partial}_{q+1}^{L,K})^* = \Sch(\Delta_{q,\mathrm{up}}^L, f_q(\ker(\partial_q^K)))$. Then,
\begin{align*}
\iota_{\Ima(f_q)}\circ \Sch(\Delta_{q,\mathrm{up}}^L, f_q(\ker(\partial_q^K)))\circ \proj_{f_q(\ker(\partial_q^K))} =& 
\underbrace{\iota_{\Ima(f_q)}\circ \Bar{\partial}_{q+1}^{L,K}}_{\partial_{q+1}^{L,K}} \circ \underbrace{(\Bar{\partial}_{q+1}^{L,K})^* \circ \proj_{f_q(\ker(\partial_q^K))}}_{(\partial_{q+1}^{L,K})^*}  \\
=&\partial_{q+1}^{L,K}\circ (\partial_{q+1}^{L,K})^* = \Delta_{q,\mathrm{up}}^{K,L}\qedhere
\end{align*}
\end{proof}

We now only require one more lemma before proving~\autoref{persBetti}.

\begin{lemma}\label{schurKernel}
Let $L: V\to V$ be self-adjoint positive semi-definite and let $W\subseteq V$ be a subspace. Then, $\ker(\Sch(L, W)) = \proj_W(\ker(L))$.
\end{lemma}

\begin{proof}
It suffices to prove this in a matrix representation by~\autoref{prop:well-defn}. Assume we have $[L] =\begin{pmatrix}
A & B \\
B^\mathrm{T} & D
\end{pmatrix}$ for some orthonormal bases for $W$ and $W^\perp$. Assume $\begin{pmatrix}
w \\
w^\perp
\end{pmatrix} \in \ker([L])$. Then, we have that $A w + B w^\perp = 0$ and $B^\mathrm{T} w + D w^\perp = 0$. From the second equality, one can get that $D^\dagger B^\mathrm{T} w = - D^\dagger D w^\perp$. Thus,
\[
(A - B D^\dagger B^\mathrm{T})w = A w - B (D^\dagger B^\mathrm{T} w) = A w + B D^\dagger D w^\perp = A w + B w^\perp = 0
\]
where the second last equality follows from~\cite[Theorem 16.1]{geomethods} by using the fact that $(I-DD^\dagger)B^\mathrm{T}=0$ and that $(D^\dagger)^\mathrm{T} = D^\dagger$. Thus, $\proj_W(\ker(L)) \subseteq \ker(\Sch(L,W))$.
For the other containment, if $w\in \ker(\Sch(L,W))$, one can check that $\begin{pmatrix}
w \\
-D^\dagger B^\mathrm{T} w
\end{pmatrix} \in \ker(L)$ by using the fact that $(I-DD^\dagger)B^\mathrm{T}=0$ again.
\end{proof}

\begin{proof}[Proof of~\autoref{persBetti}]
First, we need the following elementary linear algebra fact:

\begin{claim}[{\cite[Theorems 5.2 and 5.3]{hodgeLaplacianGraphs}}]\label{claim:lin-alg-fact}
Let $X, Y$ and $Z$ be finite dimensional inner product spaces and let $L_1: X\to Y$ and $L_2 : Y \to Z$ be linear maps. Assume that $L_2 L_1 = 0$. Then, 
\[ 
\ker(L_1 L_1^* + L_2^* L_2) = \ker(L_1^*) \cap \ker(L_2) \cong \ker(L_2) / \Ima(L_1).
\]
\end{claim}
Let $X = \ChUp{L}{K}, Y = \Ima{(f_q)}, Z = \ChDown{K}{L}$, $L_1 = \partial_{q+1}^{L,K}$ and $L_2 = (\delta_{q-1}^{K,L})^* \circ \hat{f}_q^{-1}$. Observe that 
\[
\ker((\delta_{q-1}^{K,L})^*) = \ker(\delta_{q-1}^{K,L} \circ (\delta_{q-1}^{K,L})^*) = \proj_{\ker(f_q)^\perp} (\ker(\Delta_{q,\mathrm{down}}^{K})),
\]
where the last equality follows from~\autoref{schurKernel}, as we have that 
\[\delta_{q-1}^{K,L} \circ (\delta_{q-1}^{K,L})^* = \Sch(\Delta_{q,\mathrm{down}}^K, \ker(f_q)^\perp)\]
from~\autoref{prop:schur-restriction}. Moreover, $\ker(\Delta_{q,\mathrm{down}}^K) = \ker((\partial_q^K)^* \circ \partial_q^K) = \ker(\partial_q^K)$. Thus,  
\[
f_q(\ker(\partial_q^K)) = f_q(\proj_{\ker(f_q)^\perp} (\ker(\Delta_{q,\mathrm{down}}^{K}))) = f_q(\ker((\delta_{q-1}^{K,L})^*)) = \hat{f}_q (\ker((\delta_{q-1}^{K,L})^*)).
\]
Therefore, $\Ima(L_1) \subseteq f_q(\ker(\partial_q^K)) = \hat{f}_q(\ker((\delta_{q-1}^{K,L})^*)) = \ker((\delta_{q-1}^{K,L})^* \circ \hat{f}_q^{-1}) =  \ker(L_2)$. Thus, $L_2 L_1 =0$. Hence, by applying~\autoref{claim:lin-alg-fact}, we get
\begin{align*}
    \ker(\Delta_q^{K,L}) \cong&  \ker(L_2) / \Ima(L_1) \\
                         =& \ker(\Delta_{q,\mathrm{down}}^{K,L}) / \Ima{(\partial_{q+1}^{L,K})} \\
                         =& \ker(\hat{f}_q \circ \Sch(\Delta_{q,\mathrm{down}}^K, \ker(f_q)^\perp) \circ \hat{f}_q^{-1}) / \Ima{(\partial_{q+1}^{L,K})} \\
                         =& f_q(\ker(\Sch(\Delta_{q,\mathrm{down}}^K, \ker(f_q)^\perp))) / (\Ima(\partial_{q+1}^L) \cap f_q(\ker(\partial_q^K))) \\
                         =& f_q(\ker(\partial_q^K)) / (\Ima(\partial_{q+1}^L) \cap f_q(\ker(\partial_q^K))).
\end{align*}
where the last equality follows from~\autoref{schurKernel}.
Let $(f_q)_\# : H_q(K) \to H_q(L)$ be the induced map on the homology groups. Then, 
\begin{align*}
    \Ima((f_q)_\#) =& \big(f_q(\ker(\partial_q^K)) + \Ima(\partial_{q+1}^L)\big) / \Ima(\partial_{q+1}^L) \\
                 \cong& f_q(\ker(\partial_q^K)) / (\Ima(\partial_{q+1}^L) \cap f_q(\ker(\partial_q^K))) \\
                 \cong& \ker(\Delta_q^{K,L}).
\end{align*}
Hence, $\beta_q^{K,L} = \mathrm{nullity}(\Delta_q^{K,L})$.
\end{proof}

\section{Details for Section \ref{sec:matrix-and-algo}}
\begin{proof}[Proof of~\autoref{prop:downMatrix}]
The matrix representation of the map $f_q : \ker(f_q)^\perp \to \Ima(f_q)$ with respect to the bases $\mathcal{C}$ and the canonical basis of $\Ima(f_q)$ is given by $W_{\Ima(f_q)}$. Therefore, the matrix representation of $\Delta_{q,\mathrm{down}}^{K,L} = f_q \circ \Sch(\Delta_{q,\mathrm{down}}^K, \ker(f_q)^\perp) \circ (f_q)^{-1} $ is given by $W_{\Ima(f_q)} (X - Y T^\dagger Z) W_{\Ima(f_q)}^{-1}$.
\end{proof}

\begin{proof}[Proof of~\autoref{prop:upMatrix}]
By \autoref{prop:schur-restriction}, $\Delta_{q,\mathrm{up}}^{K,L} = \iota_{\Ima(f_q)} \circ \Sch(\Delta_{q,\mathrm{up}}^L, f_q(\ker(\partial_q^K))) \circ \proj_{f_q(\ker(\partial_q^K))}$, where $\iota_{\Ima(f_q)} : f_q(\ker(\partial_q^K)) \hookrightarrow \Ima(f_q)$ is the inclusion map and $\proj_{f_q(\ker(\partial_q^K))} : \Ima(f_q) \to \proj_{f_q(\ker(\partial_q^K))}$ is the projection map. Thus, this explains the matrix in the middle of~\autoref{eq:upLapMat}. The other two matrices in~\autoref{eq:upLapMat} are for changing the basis back to the canonical basis of $\Ima(f_q)$.
\end{proof}

\begin{proof}[Proof of~\autoref{lem:fq ker = ker lap}]
By~\autoref{schurKernel} and the fact that $\ker(\partial_q^K)= \ker(\Delta_{q,\mathrm{down}}^K)$, we have
\begin{align*}
    \ker(\Delta_{q,\mathrm{down}}^{K,L}) =& \ker(\hat{f}_q \circ \Sch(\Delta_{q,\mathrm{down}}^K, \ker(f_q)^\perp)\circ \hat{f}_q^{-1}) \\
                                =& \hat{f}_q (\ker(\Sch(\Delta_{q,\mathrm{down}}^K, \ker(f_q)^\perp) \\
                                =& f_q (\proj_{\ker(f_q)^\perp}(\ker(\Delta_{q,\mathrm{down}}^K)) \\
                                =& f_q (\proj_{\ker(f_q)^\perp}(\ker(\partial_q^K)) \\
                                =& f_q (\ker(\partial_q^K)).\qedhere
\end{align*}
\end{proof}

\section{Details for Section \ref{sec:mono}}

To prove~\autoref{prop:surj-mon}, we need the following lemma.
\begin{lemma}\label{lem:extension}
Let $f: U \to V$ be a linear map between inner product spaces $U$ and $V$. Assume that $U$ is a subspace of an inner product space $\hat{U}$ and that $f$ extends to a map $\hat{f}: \hat{U} \to V$. Then, $\hat{f} \circ \hat{f}^* \succeq f \circ f^*$.
\end{lemma}

\begin{proof}
Let $U^\perp$ be the orthogonal complement of $U$ inside $\hat{U}$. Then, $\hat{f} = f \oplus f^\perp$ where $f^\perp$ maps $U^\perp$ into V. Then, $\hat{f}\circ \hat{f}^* = f\circ f^* + f^\perp \circ (f^\perp)*$. Thus, $\hat{f}\circ \hat{f}^* - f\circ f^* = f^\perp \circ (f^\perp)^* \succeq 0$.
\end{proof}

\begin{proof}[Proof of~\autoref{prop:surj-mon}]
As $f$ and $g$ are surjective and weight preserving, we observe that $f_q(c_q^{\tau, g\circ f}) = c_q^{\tau, g}$ for every $[\tau] \in \mathcal{S}_q^M$. This implies that $\tilde{f}_q := (f_q)|_{\ker(g_q\circ f_q)^\perp} : \ker(g_q \circ f_q)^\perp \to \ker(g_q)^\perp$ is an isometric isomorphism and $(\tilde{f}_q)^{-1} = (\tilde{f}_q)^* = (f_q^*)|_{\ker(g_q)^\perp}$.

Now, let $c \in \ChDown{L}{M}$. Then, $(\partial_q^K)^* (f_{q-1})^* (c) = (f_q)^*(\partial_{q-1}^L)^* (c) = (\tilde{f}_1)^{-1} (\partial_{q-1}^L)^* (c) \in \ker(g_q\circ f_q)^\perp$. Thus, $(f_{q-1})^*(c) \in \ChDown{K}{M}$. Therefore, $(f_{q-1})^*|_{\ChDown{L}{M}} : {\ChDown{L}{M}} \hookrightarrow \ChDown{K}{M}$ is an isometric embedding as $(f_{q-1})^*$ itself is an isometric embedding by surjectivity and weight preserving property of $f$. Then, we have the following commutative diagram.
\begin{center}

\begin{tikzcd}
\ker(g_q\circ f_q)^\perp \arrow[dd, "\tilde{f}_q"', bend right]                     &  & \ChDown{K}{M} \arrow[ll, "{\delta_{q-1}^{K,M}}"]                                               \\
                                                                                    &  &                                                                                                \\
\ker(g_q)^\perp \arrow[uu, "\tilde{f}_q^{-1}"', bend right] \arrow[d, "\hat{g}_q"'] &  & \ChDown{L}{M} \arrow[uu, "(f_q^*)|_{\ChDown{L}{M}}"', hook] \arrow[ll, "{\delta_{q-1}^{L,M}}"] \\
C_q^M                                                                               &  &                                                                                               
\end{tikzcd}
\end{center}
That is, $\hat{g}_q \circ \delta_{q-1}^{L,M} = \hat{g}_q \circ \tilde{f}_q \circ \delta_{q-1}^{K,M} \circ (f_q^*)|_{\ChDown{L}{M}}$. In other words, $\hat{g}_q \circ \delta_{q-1}^{L,M} : \ChDown{L}{M} \to C_q^M$ extends to $\hat{g}_q \circ \tilde{f}_q \circ \delta_{q-1}^{K,M} = \widehat{(g\circ f)}_q \circ \delta_{q-1}^{K,M} : \ChDown{K}{M} \to C_q^M$. Then, by~\autoref{lem:extension}, we get that $\Delta_{q,\mathrm{down}}^{K,M} \succeq \Delta_{q,\mathrm{down}}^{L,M}$.
\end{proof}

\begin{proof}[Proof of~\autoref{thm:down-mono}]
Observe that  $\Delta_{q,\mathrm{down}}^{K,M} = \Delta_{q,\mathrm{down}}^{K, \Ima(g \circ f)} \succeq \Delta_{q,\mathrm{down}}^{\Ima(f), \Ima(g\circ f)}$  by~\autoref{rmk:image} and~\autoref{prop:surj-mon}. The assumption that $f$, $g$ and $g \circ f$ are weight preserving guarantees that if $[\tau] \in \mathcal{S}_q^L$, then $g_q([\tau]) \notin \Ima(g_q \circ f_q)$. This implies that the matrix representation of $\Delta_{q,\mathrm{down}}^{\Ima(f), \Ima(g\circ f)}$ will be a principal submatrix of $\Delta_{q,\mathrm{down}}^{L, M}$ 

Thus, $\Delta_{q,\mathrm{down}}^{\Ima(f), \Ima(g\circ f)}$ has larger eigenvalues than $\Delta_{q,\mathrm{down}}^{L, M}$ by Eigenvalue Interlacing Theorem~\cite[Theorem 1]{eval-interlace}. 
Hence,
\[
\lambda_{q,\mathrm{down},k}^{K,M} = \lambda_{q,\mathrm{down},k}^{K, \Ima(g \circ f)} \geq \lambda_{q,\mathrm{down},k}^{\Ima(f), \Ima(g\circ f)} \geq \lambda_{q,\mathrm{down},k}^{L, M}.
\]

To prove the other inequality, namely $\lambda_{q,\mathrm{down}, k}^{K,M} \geq \lambda_{q,\mathrm{down}, k}^{K,L}$, we will use the alternative definition of down persistent Laplacian as described in~\autoref{rem:alternative-defn}, namely $\delta_{q-1}^{K,M} \circ (\delta_{q-1}^{K,M})^*$ and $\delta_{q-1}^{K,L} \circ (\delta_{q-1}^{K,L})^*$, as they have the same spectrum as the originally defined $\Delta_{q,\mathrm{down}}^{K,M}$ and $\Delta_{q,\mathrm{down}}^{K,L}$ respectively. By~\autoref{prop:schur-restriction}, 
\begin{align*}
    \delta_{q-1}^{K,M} \circ (\delta_{q-1}^{K,M})^* =& \Sch(\Delta_{q,\mathrm{down}}^K, \ker(g_q \circ f_q)^\perp), \\ 
    \delta_{q-1}^{K,L} \circ (\delta_{q-1}^{K,L})^* =& \Sch(\Delta_{q,\mathrm{down}}^K, \ker(f_q)^\perp).
\end{align*}

As $\ker(g_q \circ f_q)^\perp \subseteq \ker(f_q)^\perp$, it follows that $\Sch(\Delta_{q,\mathrm{down}}^K, \ker(g_q \circ f_q)^\perp)$ has larger eigenvalues than $\Sch(\Delta_{q,\mathrm{down}}^K, \ker(f_q)^\perp)$ by Eigenvalue Interlacing Theorem for Schur complements~\cite[Lemma 4.5]{memoli2022persistent}. Hence, $\lambda_{q,\mathrm{down}, k}^{K,M} \geq \lambda_{q,\mathrm{down}, k}^{K,L}$.
\end{proof}

\begin{proof}[Proof of~\autoref{thm:mono-up-ess}]
The eigenvalues in the statement of the theorem correspond to the operators $\Sch(\Delta_{q,\mathrm{up}}^M, g_q(f_q(\ker(\partial_q^K))))$ and $\Sch(\Delta_{q,\mathrm{up}}^M, g_q(\ker(\partial_q^L)))$. We then observe that $g_q(f_q(\ker(\partial_q^K))) \subseteq g_q(\ker(\partial_q^L))$. Thus, by eigenvalue interlacing property of Schur complement~\cite[Lemma 4.5]{memoli2022persistent}, we obtain that $\Sch(\Delta_{q,\mathrm{up}}^M, g_q(f_q(\ker(\partial_q^K))))$ has larger eigenvalues than $\Sch(\Delta_{q,\mathrm{up}}^M, g_q(\ker(\partial_q^L)))$. Thus, $\lambda_{q,\mathrm{up},k}^{K,M,\mathrm{ess}}\geq \lambda_{q,\mathrm{up},k}^{L,M,\mathrm{ess}}$.
\end{proof}

\begin{proof}[A direct proof of~\autoref{thm:up-mono-inc} via~\autoref{thm:mono-up-ess}]
Assume $ K\hookrightarrow L \hookrightarrow M$ are inclusions. We will only prove that $\lambda_{q,\mathrm{up},k}^{K,M} \geq \lambda_{q,\mathrm{up},k}^{L,M}$. Assume that $\dim(\ker(\partial_q^K))=n$. Let $m = n_q^L - n_q^K$. Then, $\dim(\ker(\partial_q^L)) \leq n+m$. Then $\Delta_{q,\mathrm{up}}^{K,M}$ has $n_q^K - n$ essential $0$ eigenvalues, and $\Delta_{q,\mathrm{up}}^{L,M}$ has $n_q^L - \dim(\ker(\partial_q^L))$. Then
\[
n_q^K - n = n_q^L-n-m \leq n_q^L - \dim(\ker(\partial_q^L))
\]
Thus, $\Delta_{q,\mathrm{up}}^{L,M}$ has more essential $0$ eigenvalues than $\Delta_{q,\mathrm{up}}^{K,M}$. Combining this with~\autoref{thm:mono-up-ess}, we conclude that $\lambda_{q,\mathrm{up},k}^{K,M} \geq \lambda_{q,\mathrm{up},k}^{L,M}$. 
\end{proof}

\begin{figure}[H]
    \centering
    \includegraphics[scale=1]{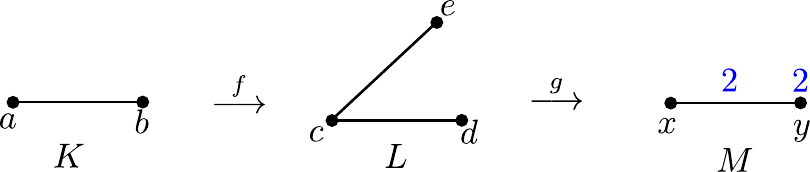}
    \caption{Example of two weight preserving simplicial where their composition is not. Weights on $K$ and $L$ are all $1$. The vertex $x$ has weight $1$, the edge $xy$ and the vertex $y$ have weight $2$. $f$ is given by $a\mapsto c$, $b\mapsto d$. $g$ is given by $c\mapsto x$, $e\mapsto y$, $d\mapsto y$}.
    \label{fig:ex4}
\end{figure}

\begin{example}\label{ex:compositionNotWP}
For the simplicial maps $f$ and $g$ depicted in~\autoref{fig:ex4}, we have that $f$ is weight preserving as it is inclusion and the simplicial complexes $K$ and $L$ has weights all $1$. The map $g$ is weight preserving because $cd$ and $ce$ both have weight $1$ and they are mapped to $xy$, which has weight $2$. Similarly, the vertices $e$ and $d$ have weight $1$ and they collapsed to the vertex $y$ which has weight $2$. However, the composition $g\circ f$ is inclusion and it maps the edge $ab$, which has weight $1$, to the edge $xy$, which has weight $2$. Thus, $g\circ f$ is not weight preserving.
\end{example}

\end{document}